\documentclass[12pt]{article}
\usepackage{amsmath, amsfonts, amssymb,amsthm, amscd}

\usepackage{pictexwd,dcpic}
\usepackage{graphicx}
\usepackage{epsf}
\usepackage{epsfig}
\usepackage{verbatim}
\usepackage{psfrag}
\usepackage{color}
\usepackage{hyperref}

\definecolor{blue}{rgb}{0,0,1}
\definecolor{red}{rgb}{1,0,0}
\definecolor{green}{rgb}{0,.7,0}                                                

\newcommand{\red}[1]{{\color{red}#1}}
                                       
\newtheorem{theorem}{Theorem}[section]
\newtheorem{``theorem''}[theorem]{``Theorem''}
\newtheorem{lemma}[theorem]{Lemma}
\newtheorem{proposition}[theorem]{Proposition}
\newtheorem{corollary}[theorem]{Corollary}

\theoremstyle{definition}
\newtheorem{definition}[theorem]{Definition}
\newtheorem{example}[theorem]{Example}

\newtheorem{question}[theorem]{Question} 			

\theoremstyle{remark}
\newtheorem{remark}[theorem]{Remark}

\newtheoremstyle{thm}
  {12pt}
  {12pt}
  {\itshape}
  {\parindent}
  {\scshape}
  {.}
  {5pt}
  {}

\theoremstyle{thm}
\newtheorem*{T6.17}{Theorem \ref{gradientthm1}}
\newtheorem{thm}{Theorem}[section]

\newtheoremstyle{prop}
  {12pt}
  {12pt}
  {\itshape}
  {\parindent}
  {\scshape}
  {.}
  {5pt}
  {}

\theoremstyle{prop}

\newtheoremstyle{lem}
  {12pt}
  {12pt}
  {\itshape}
  {\parindent}
  {\scshape}
  {.}
  {5pt}
  {}

\theoremstyle{lem}
\newtheorem*{L2.6}{Lemma \ref{lem2.4}}
\newtheorem{lem}[thm]{Lemma}

\newtheoremstyle{defn}
  {12pt}
  {12pt}
  {\itshape}
  {\parindent}
  {\scshape}
  {.}
  {5pt}
  {}

\theoremstyle{defn}
\newtheorem{defn}[thm]{Definition}

\newtheoremstyle{examp}
  {12pt}
  {12pt}
  {}
   {\parindent}
  {\scshape}
  {.}
  {5pt}
  {}

\theoremstyle{examp}

\newtheoremstyle{cor}
  {12pt}
  {12pt}
  {\itshape}
  {\parindent}
  {\scshape}
  {.}
  {5pt}
  {}

\theoremstyle{cor}

\newtheoremstyle{recipe}
  {12pt}
  {12pt}
  {\itshape}
   {\parindent}
  {\scshape}
  {.}
  {5pt}
  {}

\theoremstyle{recipe}

\newtheoremstyle{rem}
  {12pt}
  {12pt}
  {}
   {\parindent}
  {\scshape}
  {.}
  {5pt}
  {}

\theoremstyle{rem}

\newtheoremstyle{quest}
  {12pt}
  {12pt}
  {\itshape}
  {\parindent}
  {\scshape}
  {.}
  {5pt}
  {}

\theoremstyle{quest}

\newcommand{\bp}{\begin{proof}}
\newcommand{\ep}{\end{proof}}

\renewcommand{\epsilon}{\varepsilon}

\newcommand{\Fix}{\operatorname{Fix}}

\newcommand{\Z}{{\mathbb Z}}

\newcommand{\id}{\operatorname{id}}

\newcommand{\R}{{\mathbb R}}

\providecommand{\ker}[1]{$\text{ker}\ {#1}$}

\newcommand{\N}{{\mathbb N}}
\newcommand{\Q}{{\mathbb Q}}
\newcommand{\pr}{\text{pr}}

\newcommand{\T}{Z}
\newcommand{\Rot}{\operatorname{Rot}}
\newcommand{\rot}{\operatorname{rot}}
\newcommand{\lk}{\operatorname{lk}}

\newcommand{\F}{\mathcal{F}}
\newcommand{\A}{\mathcal{A}}
\newcommand{\Diff}{\operatorname{Diff}}

\renewcommand{\u}{\tilde{u}}
\renewcommand{\v}{\tilde{v}}

\newcommand{\C}{\mathbb{C}}

\newcommand{\Ftilde}{\tilde{\mathcal{F}}}

\renewcommand{\max}{\textup{max}}

\renewcommand{\P}{\mathcal{P}}
\newcommand{\D}{\mathcal{D}}

\newcommand{\psitilde}{\tilde{\psi}}

\newcommand{\Atilde}{\tilde{A}}
\newcommand{\Jtilde}{\tilde{J}}
\renewcommand{\S}{\mathcal{S}}
\newcommand{\z}{\bar{z}}
\newcommand{\x}{\bar{x}}
\newcommand{\y}{\bar{y}}

\newcommand{\Twist}{\textup{twist}}
\newcommand{\half}{\frac{1}{2}}
\newcommand{\psibar}{\bar{\psi}}

{\catcode`"=12 \gdef\hex{"}}

\mathchardef\laplace=\hex0001

\mathchardef\nabla=\hex0272

\makeatletter

\def\@@dalembert#1#2{\setbox0\hbox{$#1\mathrm I$}

  \vrule height\ht0 depth\z@ width.04\ht0

  \rlap{\vrule height\ht0 depth-.96\ht0 width.8\ht0}

  \vrule height.1\ht0 depth\z@ width.8\ht0

  \vrule height\ht0 depth\z@ width.1\ht0 }

\def\dalembert{\mathbin{\mathpalette\@@dalembert{}}\,}

\makeatother

\date{}

\begin{document}
\title{First Steps Towards a Symplectic Dynamics}

\author{ B. Bramham\footnote{\text{\bf bramham@ias.edu}  } and H. Hofer\footnote{\text{\bf hofer@ias.edu} } \\
Institute for Advanced Study\\
Princeton, NJ 08540}
\maketitle
\tableofcontents
\section{What Should Symplectic Dynamics be?}

 Many interesting physical systems have mathematical descriptions
as finite-dimensional or infinite-dimensional Hamiltonian systems.
 According to A. Weinstein, \cite{Wein},  Lagrange was the first to notice that the dynamical systems occurring in the mathematical description
of the motion of the planets can be written in a  particular form, which we call today a Hamiltonian system.
  Poincar\'e who started the modern theory of dynamical systems
and symplectic geometry developed a particular viewpoint combining  geometric and dynamical systems ideas    in the study of Hamiltonian systems.
After Poincar\'e the field of dynamical systems and the field of symplectic geometry developed separately.
Both fields have rich theories and the time seems ripe to develop the common core with highly integrated ideas 
from both fields.  Given the state of both fields this looks like a promising undertaking.
Though it is difficult to predict what ``Symplectic  Dynamics'' ultimately will be, it is not difficult to give examples
which show how dynamical systems questions and symplectic ideas come together in a nontrivial way.

\begin{figure}[htbp]
\begin{center}
\includegraphics[trim=0 2cm 0 -1cm, clip, scale=.6]{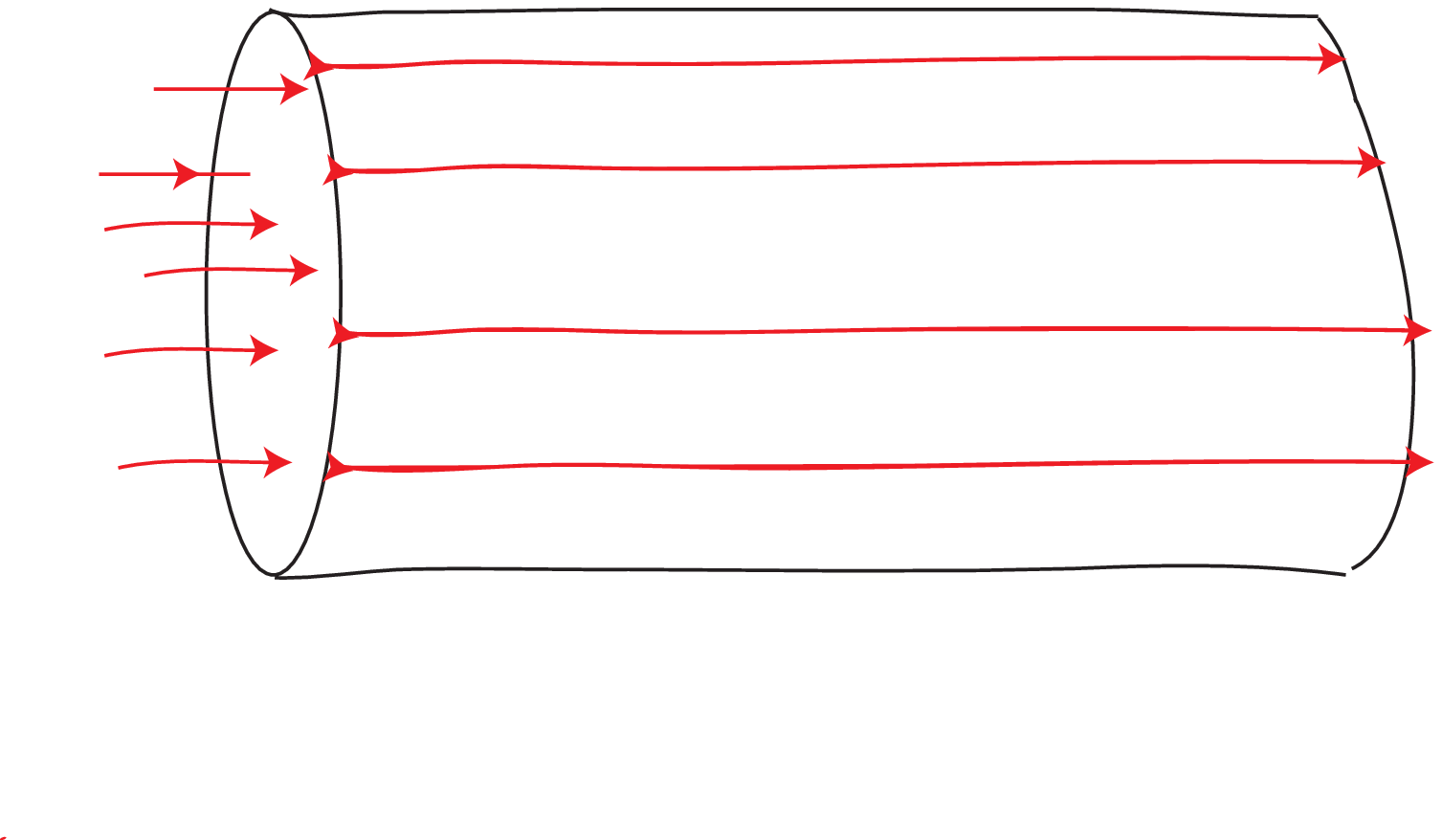}
\end{center}
\caption{A flow tube with a flow being standard at the boundary.}
\label{FIG-H1}
\end{figure}
Assume we have a cylinder, where the flow on the boundary is standard. The flow-lines enter in a standard way on the left 
and leave in a standard way on the right, see Figure \ref{FIG-H1}.  Here is the question: What must happen in the tube, assuming there are no rest points, so that not all flow-lines entering
on the left will leave on the right?

It is easy to modify the flow by introducing a pair of periodic orbits with the desired properties, see Figure \ref{FIG-H2}.

\begin{figure}[hbt]
\begin{center}
\includegraphics[height=80mm, bb=0 0 550 400]{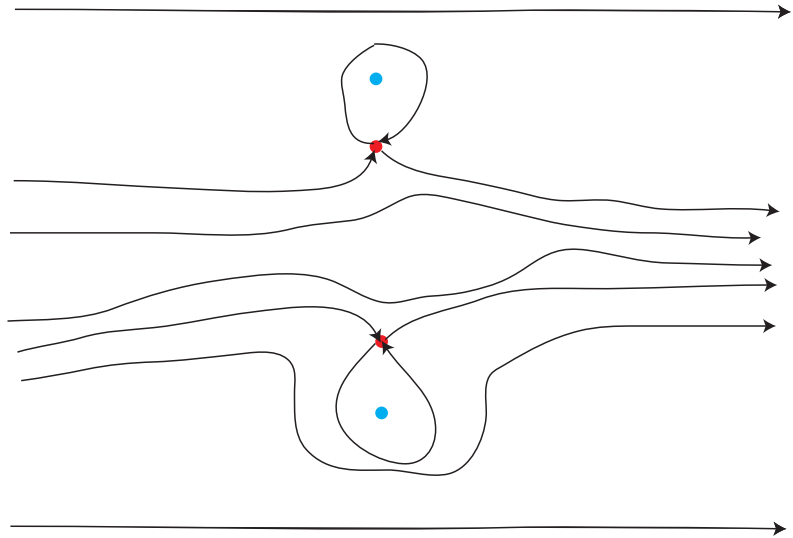}
\end{center}
\caption{Two periodic orbits are introduced in order to have a flow without rest points, but not all flow lines entering left leave on the right. }
\label{FIG-H2}
\end{figure}

On could ask next if we can achieve the desired effect without periodic orbits. That turns out to be a very hard question and is closely related to the Seifert conjecture, \cite{Seifert}.
It was solved in the category of smooth vector fields by K. Kuperberg in 1993, \cite{Kuperberg}.
One can modify the flow smoothly (even real analytically), without rest points and periodic orbits, so that not all orbits go through.
In the category of volume-preserving, G. Kuperberg, \cite{GKuperberg}, showed that the same holds on the  $C^1$-level.
  Modulo the question if G. Kuperberg's example can be made smooth  it seems that there cannot be any interesting additional contribution. However, it is precisely here 
where things become even more interesting and where we obtain a first glimpse of a ``Symplectic Dynamics''.

Fix on a compact three-manifold $M$ (perhaps with boundary) a volume form $\Omega$. A vector field $X$ is volume preserving provided $L_X\Omega=0$.
By the Cartan homotopy formula this means that 
$$
0 = i_X d\Omega + d i_X\Omega = d(i_X\Omega).
$$
Let us assume for the moment that $H^1(M)=H^2(M)=0$. Then we find a 1-form $\Gamma_0$ with
 $$
 i_X\Omega=d\Gamma_0
 $$
 and any 1-form $\Gamma$ with that property can be written as 
 $$
 \Gamma=\Gamma_0 +dh
 $$
 for a smooth map $h$.
 
Consider the collection ${\mathcal V}$ of all smooth nowhere vanishing $\Omega$-preserving vector fields.
Among these there is the interesting subset ${\mathcal V}^\ast$ consisting of those vector fields $X$
for which there exists a $\Gamma$ with $d\Gamma=i_X\Omega$, and so that $\Gamma(X)>0$. Observe that if $X'\in {\mathcal V}$ is close to $X$ we still have that $i_{X'}\Omega = d\Gamma'$
for a $\Gamma'$ close to $\Gamma$ and consequently  $\Gamma'(X')>0$ (by the compactness of $M$).
So we see that ${\mathcal V}^\ast$ is open in ${\mathcal V}$, in for example the $C^1$-topology, provided 
$H^1(M)=H^2(M)=0$.

\begin{definition}  Let $(M,\Omega)$ be a closed three-manifold equipped with a volume form. An $\Omega$-preserving vector field $X$ is called \emph{Reeb-like} provided there exist
a one-form $\Gamma$ satisfying $d\Gamma=i_X\Omega$ and $\Gamma(X)>0$ at all points
of $M$.
\end{definition}

  Assume that $X$ is Reeb-like,
so that $\Gamma(X)>0$ for some $\Gamma$ with $d\Gamma=i_X\Omega$. Note that this implies that $\Gamma\wedge d\Gamma$ is a volume form.
In other words $\Gamma$ is a contact form.

Define a positive function $f$ by
$$
f=\frac{1}{\Gamma(X)}.
$$
Then $Y=fX$ satisfies $\Gamma(Y)=1$ and $d\Gamma(Y,.) = i_Y d\Gamma=f i_Xi_X\Omega=0$. In particular
$$
L_Y(\Gamma\wedge d\Gamma) = 0.
$$
So $Y$ is Reeb-like for the modified volume form $\Gamma\wedge d\Gamma$, but satisfies the stronger condition
$\Gamma(Y)=1$. Observe that $Y$ and $X$ have the same unparameterized flow lines. So for many questions 
one can study $Y$ rather than $X$.

\begin{definition}
Let $M$ be a compact three-manifold. A \emph{Reeb vector field} on $M$ is a vector field for which there exists a contact form $\lambda$
with $\lambda(X)=1$ and $d\lambda(X,.)=0$.
\end{definition}

In \cite{EH} it was shown that if for a Reeb like vector field not all orbits pass through, then there exists a periodic orbit.
More precisely the method of proof in this paper shows the following result, where $D$ is the closed unit disk in ${\mathbb R}^2$.
\begin{theorem}[Eliashberg-Hofer]
Let $Z = [0,1]\times D$ with coordinates $(z,x,y)$ equipped with a contact form $\lambda$ which near $z=0$ or $z=1$ and 
$x^2+y^2=1$ has the form $dz + xdy$, so that close to the boundary the associated Reeb vector field is given by $(1,0,0)$. Then, if not all entering orbits go through, there has to be a periodic orbit
inside $Z$.
\end{theorem}
In other words, complicated Reeb dynamics produces periodic orbits. 
But even much more is true as we shall see. There is a holomorphic curve
theory, in the spirit of \cite{G},  related to the dynamics of Reeb-like vector fields, see \cite{H1}. The holomorphic curves allow to quantify the
complexity of the dynamics in terms of periodic orbits and relations between them. The latter are again expressed in terms of holomorphic curves.
Symplectic field theory (SFT), \cite{EGH}, uses the same ingredients to derive contact and symplectic invariants. However, it is possible to shift the focus
onto the dynamical aspects. The already strongly developed SFT gives an idea of the possible richness of the theory one might expect.
 This is precisely the key observation which indicates that there should be a field accurately described
as ``Symplectic Dynamics'' with ideas and techniques based on the close relationship between dynamics and associated 
holomorphic curve theories, as they occur in symplectic geometry and topology.   Our paper describes some of the observations.

\section{Holomorphic Curves}
In the first subsection we introduce the holomorphic curve theory associated to a contact form on a three-dimensional manifold.
This can also be done in higher dimensions. However, we shall restrict ourselves to low dimensions. Here the results which can be obtained
look the strongest.
\subsection{Contact Forms and Holomorphic Curves}\label{S:contact_forms_and_holomorphic_curves}
Consider a three-manifold $M$ equipped with a  contact form  $\lambda$.
The Reeb vector field associated to $\lambda$ is denoted by $X$ and, as previously explained, defined by
$$
i_X\lambda =1\ \ \text{and}\ \ \ \ i_Xd\lambda=0.
$$
There is another piece of data associated to $\lambda$. Namely the \emph{contact structure} $\xi$ defined as the kernel bundle associated to $\lambda$.
The form $d\lambda$ defines on the fibers of $\xi\rightarrow M$ a symplectic structure. Consequently, $\lambda$ gives us a canonical way to split the tangent space
$TM$ of $M$ into a line bundle $L$ with preferred section $X$ and a symplectic vector bundle $(\xi,d\lambda)$:
$$
TM\equiv (L,X)\oplus(\xi,d\lambda).
$$
We can pick a complex structure $J$ for $\xi$, so that $d\lambda(h,Jh)>0$ for $h\neq 0$. Then we can extend $J$ to an ${\mathbb R}$-invariant
almost complex structure $\tilde{J}$ on ${\mathbb R}\times M$ by requiring that the standard tangent vector $(1,0)$ 
at $(a,m)\in {\mathbb R}\times M$ is mapped to $(0,X(m))$.
\begin{figure}[hbt]
\begin{center}
\includegraphics[height=40mm,bb=0 0 512 370]{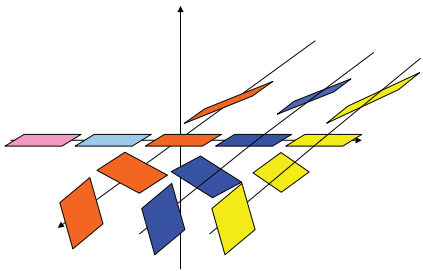}
\end{center}
\caption{A contact structure. }
\label{FIG-H3}
\end{figure}
At this point we have equipped
${\mathbb R}\times M$ with an ${\mathbb R}$-invariant almost complex structure that couples the Reeb vector field 
with the $\R$-direction.  We will refer to such as an almost complex structure \emph{compatible} with $\lambda$.  

It is natural to ask about the existence of holomorphic maps from Riemann surfaces with image in 
${\mathbb R}\times M$ and raise the question whether the geometry of these curves reflect in some way the 
dynamics of $X$, since $\tilde{J}$ couples the Reeb vector field with the ${\mathbb R}$-direction.
That in fact turns out to be true and was used by the second author to prove certain cases of the Weinstein conjecture, \cite{H1}.
This approach was in part motivated by Gromov's pseudoholomorphic curve theory for symplectic manifolds. However, the extension for contact manifolds
is by no means  straight forward, since the compactness issues for solution spaces are tricky, see \cite{H1,BEHWZ}.

The Weinstein conjecture was formulated in \cite{W78} and stipulates that on a closed manifold a Reeb vector field has a periodic orbit.
The first breakthrough came in \cite{Vi} followed by \cite{FHV}. In \cite{HV} it was shown that this conjecture can sometimes
be solved if holomorphic spheres are present and this paper was the starting point of linking the Weinstein conjecture to Gromov-Witten theory, before it really existed, 
which was completed in \cite{Tian}. The conjecture in dimension three was recently settled by Taubes,
\cite{Taubes}, and uses a relationship between Seiberg-Witten theory and holomorphic curve theory.
The conjecture  in higher dimensions is open. 
For example it is not known if every Reeb vector field on $S^5$ has a periodic orbit.
The only  result in higher dimensions, which proves the existence of periodic orbits
for a class of Reeb vector fields on every closed manifold (which admits a Reeb vector field), is given in \cite{albers-hofer}.
Let us begin with a detailed  discussion of the holomorphic curve theory.

As it turns out one should study tuples $(S,j,\Gamma,\tilde{u})$ with $(S,j)$ being a closed Riemann surface, $\Gamma$ a finite set of punctures,
and $\tilde{u}:=(a,u):S\setminus\Gamma\rightarrow {\mathbb R}\times M$ a smooth map with non-removable singularities at $\Gamma$ satisfying
the first order elliptic system
$$
T\tilde{u}\circ j = \tilde{J}\circ T\tilde{u}.
$$
This is a nonlinear Cauchy-Riemann-type equation.
It also  turns out to be useful to consider two tuples $(S,j,\Gamma,\tilde{u})$ and $(S',j',\Gamma',\tilde{u}')$ equivalent
if there exists a biholomorphic map $\phi:(S,j)\rightarrow (S',j')$ with $\phi(\Gamma)=\Gamma'$ and $\tilde{u}'\circ\phi=\tilde{u}$. We denote an equivalence class
by $[S,j,\Gamma,\tilde{u}]$.
Note that in symplectic field theory we consider a somewhat different equivalence also incorporating the natural ${\mathbb R}$-action
on ${\mathbb R}\times M$.

In a first step let us show that the dynamics of $X$ can be viewed as a part of the theory.  Given a solution $x:{\mathbb R}\rightarrow M$ of $\dot{x}=X(x)$
we can consider $[S^2,i,\{\infty\},\tilde{u}]$ with $\tilde{u}(s+it)=(s,x(t))$. Here $(S^2,i)$ is the standard Riemann sphere and $S^2\setminus\{\infty\}$ is identified with ${\mathbb C}$ with coordinates $s+it$. Observe that if $y$ is another solution of $\dot{y}=X(y)$ with $y(0)=x(t_0)$, then $[S^2,i,\{\infty\},\tilde{v}]$ with $\tilde{v}(s+it)=(s+c,y(t))$ is the same class.
Indeed take $\phi(s+it) =(s+c)+i(t+t_0)$ which defines a biholomorphic map $S^2\rightarrow S^2$ fixing $\infty$. Then
$$
\tilde{u}\circ \phi(s+it) = \tilde{u}((s+c)+i(t+t_0))= (s+c,x(t+t_0))=(s+c,y(t))=\tilde{v}(s+it).
$$
Hence
$$
[S^2,i,\{\infty\},\tilde{v}]=[S^2,i,\{\infty\},\tilde{u}].
$$
We call this particular type of class an \emph{orbit plane}, or a \emph{plane over a Reeb orbit}.  

If an orbit $x$ is periodic, say $x(t+T)=x(t)$, then it also gives us the class
$$
[S^2,i,\{0,\infty\},\tilde{u}]
$$
where $S^2\setminus\{0,\infty\}$ can be identified with ${\mathbb R}\times ({\mathbb R}/{\mathbb Z})$, and
$$
\tilde{u}(s,[t]) = (Ts,x(Tt)).
$$
We call this an \emph{orbit cylinder}, or a \emph{cylinder over a periodic Reeb orbit}.

That there is an interesting theory, which still has to be explored much further, comes from the fact
that there are many holomorphic curves which interrelate these simple building blocks.  When $M$ is compact, 
these are the curves satisfying a finite energy condition.  

\begin{figure}[hbt]
\psfrag{RtimesM}{$\R\times M$}
\psfrag{M}{$M$}
\begin{center}
\includegraphics[scale=.32]{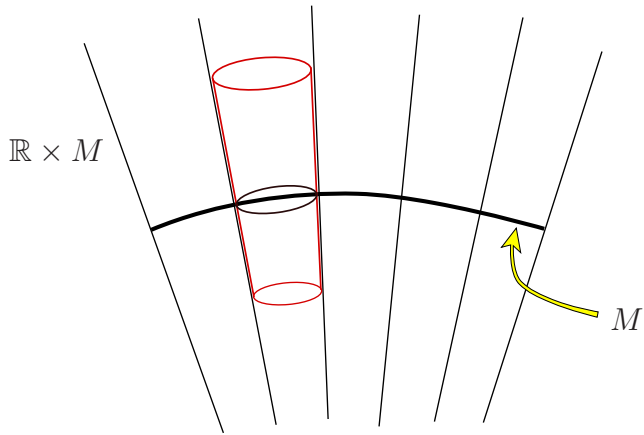}
\end{center}
\caption{A cylinder over a periodic orbit.}
\label{FIG-H4}
\end{figure}

\subsection{Notions of Energy}\label{S:notions_of_energy}
Let us introduce two important quantities, which are certain energy measurements.  

\begin{definition}
Let $[S,j,\Gamma,\tilde{u}]$ be a pseudoholomorphic curve. We assume that the punctures are non-removable.
Then we say it is a \emph{finite energy curve} provided
$$
E(\tilde{u}):=\sup_{\varphi} \int_{S\setminus\Gamma} \tilde{u}^\ast d(\varphi\lambda) <\infty.
$$
Here the supremum is taken over the collection $\Sigma$ of all smooth maps $\varphi:{\mathbb R}\rightarrow [0,1]$ 
with $\varphi'(s)\geq 0$.
\end{definition}

If we compute the energy $E$ of a cylinder over a $T$-periodic orbit we obtain the identity
$$
E=T.
$$
However, for the energy of a plane over a Reeb orbit we find
$$
 E=\infty.
$$
There is another useful energy which can be introduced.
\begin{definition}
The  \emph{$d\lambda$-energy} is defined by 
$$
E_{d\lambda}(\tilde{u})=\int_{S\setminus\Gamma} u^\ast d\lambda.
$$
\end{definition}

This energy turns out to be $0$ in both the previous cases. However, there are in general many interesting holomorphic curves
which have a positive $d\lambda$-energy. These are in fact the curves used to establish relations between the periodic orbits.
The easiest examples are finite energy planes.  
\begin{definition}
A \emph{finite energy plane} is an equivalence class $[S^2,j,\{\infty\},\tilde{u}]$ for which $\infty$ is (as usual) not removable
and $0<E<\infty$.
\end{definition}
We will see in Theorem \ref{T:feplanes_approach_periodic_orbits} that finite energy planes behave very differently 
to the (infinite) energy planes over orbits that we just encountered.  Interesting properties 
of finite energy planes were used in \cite{H1} to prove cases of the Weinstein conjecture, \cite{W78}.

It takes some analysis to show that if $\tilde{u}=(a,u):{\mathbb C}\rightarrow {\mathbb R}\times M$ represents 
a finite energy plane, then its ${\mathbb R}$-component $a$ is proper. That means that $a(z)\rightarrow \infty$ 
for $|z|\rightarrow\infty$.  Then, as a consequence of Stokes' theorem one easily verifies
\begin{lemma}
For a finite energy plane $\tilde{u}$ we have the equality 
$$
E(\tilde{u})=E_{d\lambda}(\tilde{u}).
$$
\end{lemma}

The finite energy planes have some nice properties. For example they detect contractible periodic orbits
of the Reeb vector field.

\begin{theorem}\label{T:feplanes_approach_periodic_orbits}
Assume that $\tilde{u}:=(a,u):{\mathbb C}\rightarrow {\mathbb R}\times M$ is smooth and satisfies the differential 
equation 
$$
T\tilde{u}\circ i =\tilde{J}\circ T\tilde{u}.
$$
Assume further that $\tilde{u}$ is nonconstant and $E(\tilde{u})<\infty$.
Then $T:=E(\tilde{u})\in (0,\infty)$ and 
for every sequence $r_k\rightarrow \infty$ there exists a subsequence
$r_{k_j}$ and a solution $x:{\mathbb R}\rightarrow M$ of 
$$
\dot{x}=X(x)\ \ \text{and}\ \  x(0)=x(T)
$$
so that in addition
$$
\lim_{j\rightarrow\infty} u(r_{k_j}\cdot e^{2\pi i t}) = x(Tt) \ \ \ \text{in}\ \ \ C^{\infty}({\mathbb R}/{\mathbb Z},M).
$$
\end{theorem}
In other words, non-constant solutions on the 1-punctured Riemann sphere are related to periodic orbits
for the Reeb vector field. The period in fact being the quantity
$$
T=\int_{\mathbb C} u^\ast d\lambda =E_{d\lambda}(\tilde{u}).
$$

\begin{figure}[hbt]
\psfrag{RtimesM}{$\R\times M$}
\psfrag{M}{$M$}
\begin{center}
\includegraphics[scale=.4]{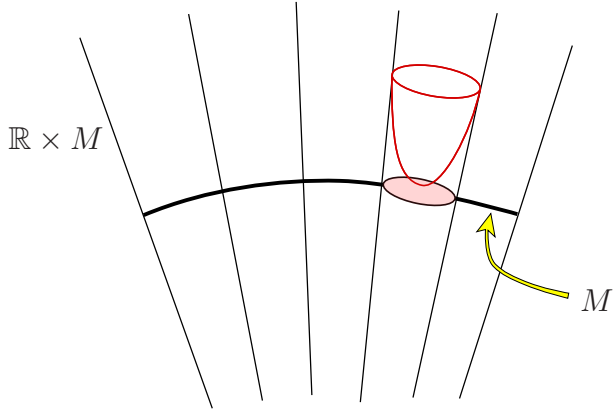}
\end{center}
\caption{A finite energy plane asymptotic to a cylinder over a periodic orbit, and the projected 
disk in the contact manifold.}
\label{FIG-H5}
\end{figure}

The main idea in \cite{H1} was to show that for the nonlinear Cauchy-Riemann problem associated to the Reeb vector field
of an overtwisted contact form there always exists a finite energy plane, showing the existence of a periodic orbit.

\subsection{Global Systems of Surfaces of Section}

Assume that $\lambda$ is a contact form on the closed three-manifold $M$. 
Suppose we have fixed a compatible almost complex structure $\tilde{J}$ on $\R\times M$ as described before.
If $[S,j,\Gamma,\tilde{u}]$ is an equivalence class of solutions associated to
the nonlinear Cauchy-Riemann equation, we can associate to it its \emph{image} $F$ defined by
$$
F_{[S,j,\Gamma,\tilde{u}]}=\tilde{u}(S\setminus\Gamma).
$$
This definition does not depend on the representative we have picked.  We call 
$[S,j,\Gamma,\tilde{u}]$ an \emph{embedded solution} provided the map
$$
\tilde{u}:S\setminus\Gamma\rightarrow {\mathbb R}\times M
$$
is an embedding. Let us also observe that for a given solution $[S,j,\Gamma,\tilde{u}]$ and real constant
$c\in {\mathbb R}$ we obtain another solution $[S,j,\Gamma,\tilde{u}]_c$ defined by
$$
[S,j,\Gamma,\tilde{u}]_c := [S,j,\Gamma,\tilde{u}_c],
$$
where $(a,u)_c = (a+c,u)$ and $\tilde{u}=(a,u)$. Observe that the image of $[S,j,M,\tilde{u}]_c$ is the image
of $[S,j,\Gamma,\tilde{u}]$ shifted by $c$ via the obvious ${\mathbb R}$-action on ${\mathbb R}\times M$.

\begin{definition}\label{D:finite_energy_foliation}
Let $\lambda$ be a contact form on the three-manifold $M$ and $\tilde{J}$ a compatible almost complex structure.  A \emph{finite energy foliation} $\tilde{\mathcal F}$ associated to this 
data is a smooth foliation of ${\mathbb R}\times M$ by the images of embedded curves, having finite 
energy, with the property that if $F$ is a leaf, each $F_c$ is also a leaf\footnote{We will 
sometimes emphasize this last property by referring to a finite energy foliation as being 
``$\R$-invariant''.  In this article we will not consider finite energy foliations without this 
invariance under the $\R$-translations.}  .
\end{definition}

Let us observe that if we drop the requirement of finite energy we always have the following object.  

\begin{definition}\label{D:vertical_foliation}
 Let $\lambda$ be a contact form on the three-manifold $M$. The \emph{vertical foliation} 
$\Ftilde^\nu(M,\lambda)$ is defined to be the 
foliation of $\R\times M$ whose leaves take the form $\R\times\phi(\R)$ over all Reeb 
trajectories $\phi:\R\rightarrow M$.  
\end{definition}

Note that the leaves of the vertical foliation are pseudoholomorphic for any almost 
complex structure compatible with $\lambda$, and they are invariant under $\R$-translations.  
However, the vertical foliation is only a finite energy foliation when every Reeb orbit is periodic.  
This simple observation will be crucial in section \ref{S:proof_of_a_theorem_of_Franks_and_Handel}.   

Finite energy foliations, when they exist, have important consequences for the Reeb flow.  
Recall that a \emph{surface of section}\footnote{{Contrast this definition with that of a global
surface of section which has the additional property that every orbit, other than the bindings or spanning orbits, hits the surface in forward and backward time.}}
 for a flow on a three-manifold is an embedded surface, possibly with 
boundary, with the property that the flow is 
transverse to the interior of the surface while each boundary circle is a periodic orbit, {called binding or spanning orbit}.  Poincar\'e used this 
notion to great effect, constructing such surfaces locally ``by hand''.  It was observed in 
\cite{HWZ-convex} that a finite energy foliation gives rise to a filling of the entire three-manifold by 
surfaces of section, simply by projecting the leaves down via the projection map
\[
 		\textup{pr}:\R\times M\rightarrow M.  
\]
The resulting filling $\F$ of 
the three-manifold was therefore called a global system of surfaces of section, which we losely define as follows. 

\begin{definition}\label{D:global_system_of_surfaces_of_section}
Let $M$ be a three-manifold with a nowhere vanishing vector field $X$ having a globally defined flow.  
A \emph{global system of surfaces of section} for this data is a finite collection of periodic orbits 
$\P$ of the flow, called the spanning orbits, and a smooth foliation of the complement 
\[
 M\backslash\P
\]
by embedded punctured Riemann surfaces $\S$, such that each leaf in $\S$ converges to a spanning orbit 
at each of its punctures, and such that the closure of each leaf in $M$ is a surface of section for the flow.  
\end{definition}

An adapted open book associated to a contact three-manifold \cite{geiges} provides a familiar example 
of a global system of surfaces of section, but in contrast to the situation we describe here, 
one only knows \emph{there exists} a Reeb flow making the leaves surfaces of section.  A further distinction, is 
that in an open book all leaves lie in a single $S^1$ family, in particular they all have the same collection of spanning orbits.  

\begin{proposition}[Hofer-Wysocki-Zehnder,\cite{HWZ-tight,HWZ-katok}]\label{P:global_system_of_surfaces_of_section}
 Suppose that a three-manifold $M$ equipped with a contact form $\lambda$ admits an associated finite energy 
foliation $\Ftilde$ (with respect to some almost complex structure compatible with $\lambda$).  Then the projection 
of the leaves down to $M$ is a singular foliation 
\[
 	\F:=\{\,\textup{pr}(F)\, |\, F\in\Ftilde\,\}
\]
with the structure of a global system of surfaces of 
section for the Reeb flow.  
\end{proposition}
{Note that in this context, where the dynamics comes from a Reeb vector field, each surface 
of section in $\F$ comes naturally equipped with an area form of finite volume, preserved by the flow.  
Indeed, that the Reeb vector field is transverse to the interior of a surface $S\in\F$ 
implies that $d\lambda$ restricts to a non-degenerate $2$-form on $S$, while the embeddedness of the leaf 
up to the boundary yields that the total volume, i.e. the $d\lambda$-energy of the corresponding 
holomorphic curve, is finite.  
In fact, by Stokes theorem the area of each leaf is equal to the sum of the 
periods of the positive punctures minus the sum of the periods of the negative punctures.}

Proposition \ref{P:global_system_of_surfaces_of_section} can be seen as follows.  If $F$ is a leaf in $\Ftilde$, then either $F$ is a {cylinder over a periodic orbit}, or it is disjoint from each of its $\R$-translates, $F_c$, $c\neq 0$.  In the former case the projection of $F$ down to $M$ is just the periodic orbit 
it spans.  In the latter case, the leaf $F$ must be nowhere tangent to $(1,0)$ in $T(\R\times M)$, and since $(1,0)$ 
is coupled by $\Jtilde$ with the Reeb vector field $X$, the leaf is transverse to the complex line $\R(1,0)\oplus\R X$. 
This amounts to the projection of the leaf being transverse to $X$ in $M$.  A variety of necessary and sufficient conditions for the projection of a curve to be embedded are given in \cite{Siefring}.

A global system of surfaces of section $\F$ inherits a certain amount of other structure from the finite energy 
foliation $\Ftilde$.  In particular, if the contact form has only non-degenerate periodic orbits, there are only two possibilities for the 
local behavior near each spanning orbit, depending on the parity of its Conley-Zehnder index.  Local cross-sections 
are illustrated in Figure \ref{F:local_cross-sections}\footnote{In general, the picture on the left of 
Figure \ref{F:local_cross-sections} could happen at a spanning orbit having even parity Conley-Zehnder index  
if the orbit has a constraint in the form of an asymptotic ``weight''.  
But in all the examples in this paper, there are only weights on odd index orbits so this doesn't happen.}.   

\begin{figure}[thb]
\begin{center}
\includegraphics[scale=.42]{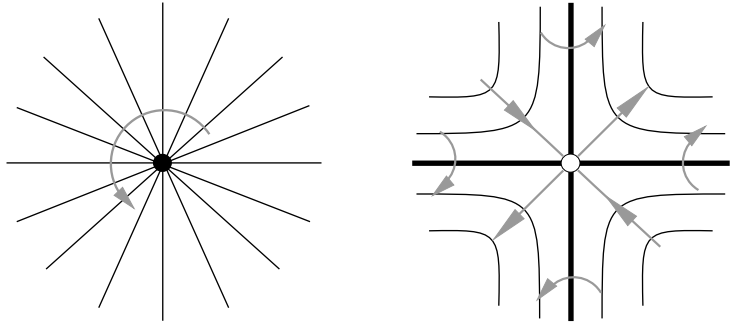}
\caption{Illustrating the two forms of behavior possible in a neighborhood of a spanning orbit in a global 
system of surfaces of section $\F$ that arises as the projection of a finite energy foliation $\Ftilde$, 
for a non-degenerate contact form.  
On the left all surfaces which enter the neighborhood converge to the spanning orbit.  This occurs when 
the spanning orbit has odd parity Conley-Zehnder index (see also the footnote on previous page).  
In the second case  precisely four leaves entering a neighborhood of the orbit connect to it.  This 
occurs whenever the spanning orbit has even Conley-Zehnder index.  The Reeb trajectories are transversal to 
the page and to the leaves in the direction of the arrows.  In an open book the picture is as on the left at 
every spanning orbit.}
\label{F:local_cross-sections}
\end{center}
\end{figure}
For a global illustration of a global system of surfaces of section on $S^3$ see Figures \ref{Fig14} and 
\ref{Fig16}, and on a solid torus see \ref{F:example_of_a_GSSS_for_a_mappingtorus}, 
\ref{F:foliaton_of_mapping_torus_and_cross-section}, 
\ref{F:different_boundary_conditions}, \ref{F:nonsimple_foliation}.  

It is not clear at all if in any given situation a finite energy foliation exists. However, it turns out, that quite often they do.  
The first such result appeared in \cite{HWZ-convex} and was generalized further in \cite{HWZ-tight}.
These papers study tight Reeb flows on $S^3$. According to a classification result every positive tight contact form
$\lambda$ on $S^3$ is, after a smooth change of coordinates, of the form $f\lambda_0$, where $f:S^3\rightarrow (0,\infty)$ is a smooth map
and $\lambda_0= \frac{1}{2}[q\cdot dp - p\cdot dq]$ is the standard contact form on $S^3$ whose associated contact structure is the
line bundle of complex lines in $TS^3$, where $S^3$ is seen as the unit sphere in ${\mathbb C}^2$ and $q+ip$ are the coordinates.
The flow lines of $X$ on $S^3$ associated to $f\lambda_0$ are conjugated to the Hamiltonian flow on the energy surface 
$$
N=\{ \sqrt{f(z)}z\ |\ |z|=1\},
$$
where we indentify ${\mathbb C}^2$ with ${\mathbb R}^4$ via $q+ip\rightarrow (q_1,p_1,q_2,p_2)$, and the latter has the standard symplectic form
$\omega= dq_1\wedge dp_1 + dq_2\wedge dp_2$. So we can formulate the results in terms of star-shaped energy surfaces in ${\mathbb R}^4$,
i.e. energy surfaces bounding domains which are star-shaped with respect to $0$.

\begin{theorem}[Hofer-Wysocki-Zehnder,\cite{HWZ-convex}]\label{T:existence_of_foliation_on_generic_tight_S^3}
Assume that $N$ bounds a strictly convex domain containing zero and is equipped with the contact form $\lambda_0|N$.
For a generic admissible complex multiplication $J$ on the associated contact structure there exists a finite energy foliation
with precisely one leaf which is a cylinder over a periodic orbit and all other leaves are finite energy planes asymptotic to it.
\end{theorem}

After projecting down to the $3$-manifold $N$, each plane-like leaf gives rise to a disk-like surface 
of section with boundary the spanning orbit{, and finite volume}.  The convexity implies that the generalized Conley-Zehnder 
index of the spanning orbit is at least $3$. This implies that the orbits nearby intersect all the 
leaves of the projected foliation transversally enough that all orbits, besides the spanning orbit, 
hit each leaf in forwards and backwards time.  Thus, fixing any leaf, we obtain a well defined return map,  which is an area-preserving diffeomorphism of the open disk {}.  By an important result of Franks we obtain that
the disk map, if it has at least two periodic points must have infinitely many periodic orbits. Hence we 
obtain the following corollary.

\begin{figure}[thb]
\centering
\psfrag{t}{$3$}
\psfrag{d}{$2$}
\psfrag{cp}{$C^+$}
\psfrag{cm}{$C^-$}
\includegraphics[scale=0.6]{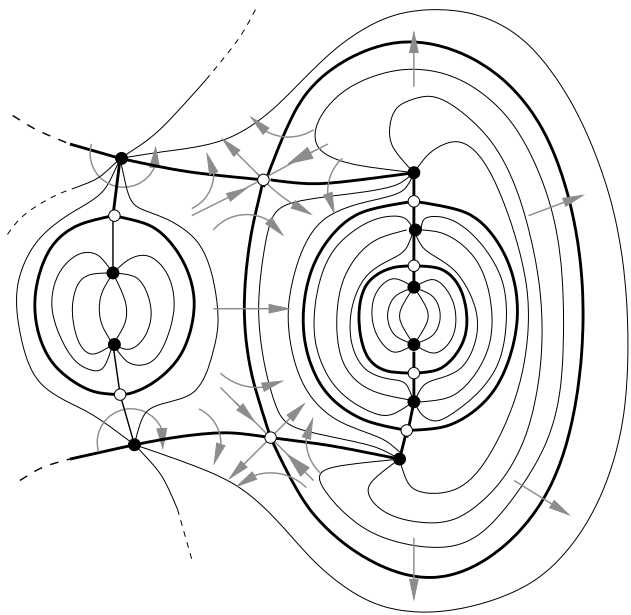}
\caption{A global system of surfaces of section of $S^3$, obtained as the  
projection of a finite energy foliation of $\R\times S^3$.  The $3$-sphere is viewed as 
${\mathbb R}^3\cup \{ \infty\}$, and the figure shows the trace of the surfaces of 
section cut by a plane.  The dots represent the spanning periodic orbits; they are 
perpendicular to the page and two dots belong to the same periodic orbit.  
The white dots represent periodic orbits of index $2$ and the black dots
periodic orbits of index $3$. The leaves are disk-like and annuli-like. The rigid 
surfaces are represented by bold curves. The grey arrows indicate the Reeb flow.}
\label{Fig14}
\end{figure}

\begin{corollary}[Hofer-Wysocki-Zehnder]
On an energy surface in ${\mathbb R}^4$ which bounds a strictly convex bounded domain, we have either precisely two geometrically distinct
periodic orbits or infinitely many.
\end{corollary}

We would like to emphasize that no kind of genericity is assumed.

{When we go to the most general case, namely that of an energy surface bounding a starshaped domain, we in general still need some
genericity assumption}\footnote{{This is only a technical assumption and one should be able to remove it. However, it might not be so easy to draw
the same conclusions about the dynamics of the Reeb vector field.}}.  For example assuming that all periodic orbits are non-degenerate and the stable and unstable manifolds
of hyperbolic orbits are transversal where they intersect.  This can always be obtained by a $C^\infty$-small perturbation of the energy surface.
Alternatively we may consider generic contact forms $f\lambda_0$ on $S^3$,  so that the associated star-shaped energy surface 
has the previously described genericity properties.

\begin{theorem}[Hofer-Wysocki-Zehnder,\cite{HWZ-tight}]
Let $\lambda=f\lambda_0$ be a generic contact form on $S^3$. Then for a generic complex multiplication on $\xi$ with associated ${\mathbb R}$-invariant
almost complex structure $\tilde{J}$ on ${\mathbb R}\times S^3$, there exists an associated finite energy foliation. Besides finitely many cylinders over periodic orbits
the other leaves are parameterized by punctured finite energy spheres with precisely  one positive puncture, but which can have several negative
punctures. The asymptotic limits at the punctures are simply covered.
\end{theorem}

\begin{figure}[thb]
\centering
\psfrag{3}{$3$}
\includegraphics[width=0.6\textwidth]{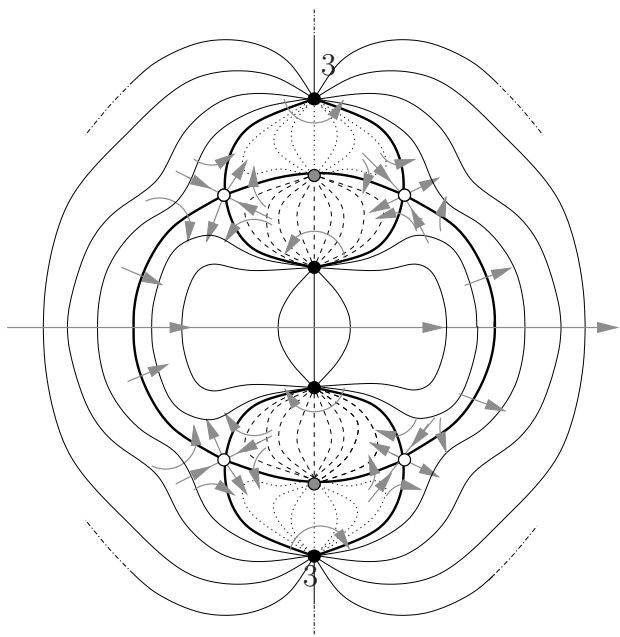}
\caption{A second example, illustrating a finite energy foliation associated to $S^3$.  
The grey dots represent a periodic orbit of index $1$, the white dots periodic orbits of index $2$ and 
the black dots periodic orbits of index $3$.  The leaves are again disk-like and annuli-like
surfaces.  The rigid surfaces are represented by bold curves. Dotted and dashed surfaces represent 
families of annuli-like surfaces connecting periodic orbits
of index $3$ with periodic orbits of index $1$. Thin curves represent
disk-like surfaces asymptotic to periodic orbits with
index $3$. The grey arrows picture the flow of the Reeb vector field.}
\label{Fig16}
\end{figure}

The precise result also contains some more technical information about the Conley-Zehnder indices of the periodic spanning orbits and we refer the reader to
\cite{HWZ-tight}.  The Figures \ref{Fig14} and \ref{Fig16}, taken from \cite{HWZ-katok}, describe some of the structure.  

From this result again one can derive that there are either two or infinitely many periodic orbits, given 
the described genericity.  We now sketch the proof of this.   

If we have only one spanning obit, then we have a \emph{global disk-like surface of section}, meaning a surface 
of section, {homeomorphic to the open disk}, with a 
well defined return map, {preserving a finite volume form.} The assertion then follows from Franks' theorem.  Assume that we have two spanning orbits. In that case we must 
have at least one hyperbolic spanning orbit
of period $T$ say.  One can show quite easily that there exists a heteroclinic chain connecting several hyperbolic spanning orbits.  Using our genericity assumption we can use symbolic dynamics
to construct infinitely many periodic orbits.  The heteroclinic chain follows immediately from the fact 
that we only have finitely many spanning orbits
and that the stable and unstable manifold of a hyperbolic orbit intersect the nearby leaves of the finite energy foliation in loops which have $\lambda$-integral
equal to the period $T$.  Essentially for area reasons the assertion follows. The reader should take any of the two Figures \ref{Fig14} or \ref{Fig16} and try to carry out the argument.

In recent papers Hryniewicz and Hryniewicz-Salomoa, \cite{Umberto,Umberto1,UmbertoPedro} have been able to give a necessary and sufficient condition for 
when a sphere-like energy surface possesses a global disk-like surface of section.  

Recent work by Albers, Frauenfelder, van Koert and Paternain,  \cite{AFKP} and Cieliebak, Frauenfelder and van Koert, \cite{CFvK}
make it feasible to use finite energy foliations in the study of the classical restricted circular planar three-body problem.

Finally let us note that finite energy foliations have important applications in contact geometry as well, see \cite{Wendl,Wendl1,Wendl2}.

\section{Holomorphic Curves and Disk Maps}\label{S:holomorphic_curves_and_disk_maps}
We saw in the last section, Theorem \ref{T:existence_of_foliation_on_generic_tight_S^3}, that finite energy 
foliations exist for generic Reeb flows on the tight $3$-sphere.  It turns out that not just one, 
but many such foliations can be constructed if we replace the $3$-sphere 
with a solid torus and restrict to generic Reeb flows which have no contractible periodic 
orbits.  

Any area preserving diffeomorphism of the disk can be put into this framework, with different iterates 
giving rise to genuinely different global systems of surfaces of section.  The question arises how one can 
profit, dynamically speaking, from this perspective. One could hope that the holomorphic curves 
provide a book-keeping tool for tracking the history and future of the orbits of a disk map.  

As a first step in this direction, we outline how these can be used to prove the Poincar\'e-Birkhoff 
fixed point theorem, and a complementary result ``Theorem'' \ref{T:two_porbits_implies_id_or_twist} 
which says that maps with two periodic orbits have some iterate which either has a ``twist'' or is 
the identity map.  This latter statement might be new.  Combining these recovers, albeit currently with an 
additional boundary condition, a celebrated theorem of Franks \cite{Franks} in the smooth category, a more recent result of Franks and Handel \cite{Franks2}, along with sharp growth estimates on periodic orbits 
known to follow from results of Le Calvez \cite{Calvez}.  

These applications due to the first author depend on more recent developments \cite{Bramham2,Bramham3} in which it is shown how to construct finite energy foliations with a prescribed spanning orbit.  
See ``Theorem''  \ref{T:existence_of_foliations:_strongest_statement} for a precise statement.  This is a 
surprising novelty, implying that in general there are many more finite energy foliations than 
might naively be expected.  This result is labelled ``theorem'' as it is not completely written up.  
Results depending on this one are also labelled in quotation marks.  

This section explains the 
existence statements for finite energy foliations, and outlines proofs of the mentioned dynamical applications.

\subsection{Reeb-like Mapping Tori} 
Let $Z_{\infty}=\R\times D$ be the infinite tube equipped with coordinates $(\z,\x,\y)$.  There is a $\Z$-action  generated by the ``$1$-shift'' automorphism $\tau(\z,\x,\y)=(\z+1,\x,\y)$.  Quotienting out by some iterate 
$\tau^n$ gives us a solid torus $\T_n$ of ``length'' $n$.   We denote by $(z+n\Z,x,y)$ the induced coordinates 
on $\T_n$, and write $z$ to mean $z+n\Z$ when the context is clear.  

\begin{definition} 
For $n\in\N$, a \emph{Reeb-like mapping torus} will refer to a contact form $\lambda_n$ on $\T_n$, for which the disk 
slice $D_0:=\{z=0\}$ is a global surface of section for the Reeb flow, and which lifts to 
a contact form $\lambda_{\infty}$ on $Z_{\infty}$ having the following properties: it is invariant under 
pull-back by the 
$1$-shift automorphism, and has contact structure $\ker\{cdz+xdy-ydx\}$, some $c>0$.  
\end{definition}

By a \emph{global} surface of section is meant that the trajectory through any point in the solid torus passes 
through $D_0$ in forwards and backwards time, and does so transversely to $D_0$.  In particular, $d\lambda_n$ 
restricts to an area form on $D_0$, and the flow induces a first return map.  

From a Reeb-like mapping torus $\lambda_n$ on $\T_n$ one can lift and project to obtain 
a sequence of mapping tori 
\[
 (\T_{1},\lambda_{1}),\ (\T_{2},\lambda_{2}),\ (\T_{3},\lambda_{3})\ldots.  
\]
The first return map of the flow on $(\T_1,\lambda_1)$ is a diffeomorphism 
$\psi:D\rightarrow D$ preserving the area form $\iota^*d\lambda_1$, where $\iota:D\hookrightarrow D_0=\{0\}\times D$ 
is inclusion, and the first return map on $(\T_n,\lambda_n)$ is then the $n$-th iterate $\psi^n$.  

In this situation, we will say that $(Z_{1},\lambda_{1})$ \emph{generates} the disk map $\psi$, and that 
$(Z_n,\lambda_n)$ generates $\psi^n$.  

\begin{lemma}\label{L:existence_of_contact_form}
Let $\psi:D\rightarrow D$ be any orientation preserving, $C^{\infty}$-diffeo-morphism preserving $dx\wedge dy$.  
Then there exists a Reeb-like mapping torus $(\T_1,\lambda_1)$ having first return map $\psi$ for which  
$\iota^{*}d\lambda_1=dx\wedge dy$, the standard Euclidean volume form.  
\end{lemma} 

A proof of this under the assumption that $\psi$ restricts to a rotation on the 
boundary of the disk is contained in \cite{Bramham}.  

Identifying $Z_n$ with $\R/n\Z\times D$, the product structure gives us a canonical way 
to assign linking numbers to pairs of homologous periodic orbits, and, along with the $S^{1}$-symmetry of 
the contact structure, canonically assign Conley-Zehnder indices to 
individual periodic orbits.  We explain this now.  

There is a canonical basis of $H_1(\partial Z_n;\Z)\simeq\Z\oplus\Z$ we will denote by $\{L_n,[\partial D]\}$, 
which are the unique elements represented by closed oriented loops of the form 
$\R/n\Z\times\{\textup{pt}\}$ and $\{\textup{pt}\}\times\partial D$ respectively.  We call $L_n$ the 
\emph{canonical longitude} on $Z_n$, and $[\partial D]$ the \emph{canonical meridian} on $Z_n$.  
With a little elementary algebraic topology, one can show that if $\gamma$ is an immersed closed loop 
in $Z_n$, disjoint from $\partial Z_n$, and homologous to $L_n$, then 
\[
			\{L_n,[\partial D]\}
\]
is also a basis for $H_1(Z_n\backslash\gamma(S^1);\Z)$ after applying the inclusion 
$\partial Z_n\hookrightarrow Z_n\backslash\gamma(S^1)$.  (This is not true if we 
replace homology groups by homotopy groups in case $\gamma$ is knotted).  Thus the following is well defined. 

\begin{definition}
Let $\gamma_{1},\gamma_{2}:S^{1}\rightarrow\T_n$ be two continously embedded closed loops having 
degree $1$ after projecting onto the $S^1$-factor of $Z_n=\R/n\Z\times D$.  Assume that their 
images are disjoint from each other and from $\partial\T_n$.  Then $\gamma_1$ determines an 
homology class in the complement of $\gamma_2$.  Define the  
\emph{linking number} to be the unique integer $\lk(\gamma_{1},\gamma_{2})$ such that $\gamma_1$ 
is homologous to 
\[
	L_n+\lk(\gamma_{1},\gamma_{2})[\partial D]\in H_1(\T_n\backslash\gamma_{2}(S^{1});\Z)
\]  
where $\{L_n,[\partial D]\}$ is the canonical basis, longitude and meridian, that we just defined.  
\end{definition}

\begin{figure}[hbt]
\psfrag{gamma1}{$\gamma_1$}
\psfrag{gamma2}{$\gamma_2$}
 \centering
 \includegraphics[scale=.41]{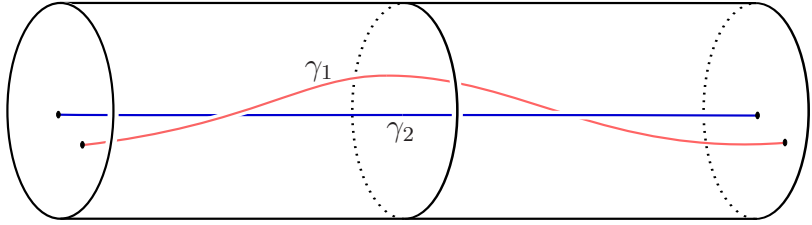}
 \caption{$\gamma_1$ and $\gamma_2$ represent two disjoint continuously embedded closed loops 
in the solid torus $Z_2$.  Their linking number is $\lk(\gamma_{1},\gamma_{2})=1$.}
 \label{F:geometrically distinct orbits}
\end{figure}

The linking number turns out to be symmetric, that is $\lk(\gamma_{1},\gamma_{2})=\lk(\gamma_{2},\gamma_{1})$. 
Moreover, with these conventions one has linear growth under iterates: if $\gamma_{i}^k:S^1\rightarrow Z_{kn}$ 
represents the unique lift of $\gamma_i$ to the longer mapping torus $(Z_{kn},\lambda_{kn})$, for 
$i\in\{0,1\}$, then $\lk(\gamma_{1}^k,\gamma_{2}^k)=k\!\cdot\!\lk(\gamma_{2},\gamma_{1})$ if $k\in\N$.  

Making use of further symmetry of the contact structure we also obtain canonical Conley-Zehnder indices 
for periodic orbits of the Reeb flow.  
Recall that usually, on a general manifold $M$ equipped with a contact form $\lambda$, at a periodic 
orbit $\gamma$ of the Reeb flow one requires a choice of symplectic trivialization, up to homotopy type, 
of the contact structure $\xi:=\ker\lambda$ along $\gamma$ to be able to assign a Conley-Zehnder index 
to $\gamma$.  In the cases at hand the contact manifold $Z_n$ is covered by $Z_{\infty}=\R\times D$ on 
which the contact structure $\xi:=\ker\lambda_{\infty}$ is invariant under the $\R$-action 
$c\cdot(\z,\x,\y)=(\z+c,\x,\y)$ (and hence also globally trivializable).  This descends to a circle 
action $c\cdot(z,x,y)=(z+c,x,y)$ on $Z_n$.  There is a unique 
homotopy class of trivializations on $Z_n$ admitting a representative which is invariant under 
this circle action.  

\begin{definition}\label{D:Conley-Zehnder_index}
Let $\gamma$ be a periodic orbit of the Reeb flow of $(Z_{n},\lambda_{n})$.  Then we take 
the \emph{Conley-Zehnder index} of $\gamma$ to be with respect to the unique homotopy class of global 
trivializations of the contact structure which admit an $S^1$-invariant representative. 
We will denote this by 
\[
							\mu(\gamma)\in\Z.  
\]
\end{definition}

\subsection{Notions of rotation number}\label{S:real_valued_rotation_numbers}
For a more sophisticated approach to rotation numbers of surface maps see 
\cite{Franks,Franks5,Franks2,LeCalvez3}, in which the area 
preserving property of the map is used to make sense of the rate at which \emph{almost all} orbits  
rotate about a given fixed point, via the Birkhoff ergodic theorem.  

For our discussion it will suffice to talk of the \emph{total} rotation number of a smooth area preserving 
diffeomorphism of the disk associated to its restriction to the boundary circle and associated to a 
periodic point, in the latter intuitively describing the infinitesimal rate at which points nearby 
rotate around the periodic orbit.  By total, we mean that the rotation numbers are real valued, as 
opposed to merely circle valued.  

Recall that any orientation preserving homeomorphism of the circle has a well defined value in $\R/\Z$ 
called its rotation number.   Any choice of lift to a homeomorphism of the real line can be assigned 
a real valued rotation number, but this depends on the choice of lift.  Similarly, any choice of 
homotopy of the circle map to the identity determines a unique lift and thus allows to assign a real 
valued rotation number.  Indeed, if $F:[0,1]\times S^1\rightarrow S^1$, is a homotopy from 
$\id=F(0,\cdot)$ to $f=F(1,\cdot)$, there is a unique lift $\bar{F}:[0,1]\times\R\rightarrow\R$ of $F$ 
satisfying $\bar{F}(0,\cdot)=\id_\R$, giving us a canonical lift $\bar{F}(1,\cdot)$ of $f$.  

More generally, if one has a diffeomorphism of a surface one can assign a circle valued rotation number 
to the restriction to any boundary component, and to each fixed point via the differential at the fixed 
point.  But we can obtain the more useful real valued rotation numbers if the map is isotopic to the 
identity and we 
choose such an isotopy class.  Moreover, this gives a canonical way to assign integer valued linking 
numbers to pairs of periodic orbits of the same period.

In the framework we are working with here, we are always in the situation of having a diffeomorphism 
of the disk for which we have chosen a mapping torus generating it as the first return map.  This is 
similar to fixing an isotopy class 
from the disk map to the identity, and indeed, the choice of a mapping torus allows us to define real 
valued rotation numbers (and linking numbers as we already saw).   Although for all of this it is 
unnecessary that the map be area preserving, we will nevertheless make use of this to make a short route 
to a workable definition.  

Consider a Reeb-like mapping torus $(Z_1,\lambda_1)$ generating a disk map $\psi:D\rightarrow D$.  
Let $f:\partial D\rightarrow \partial D$ denote the restriction of $\psi$ to the boundary.  
The choice of mapping torus gives us a canonical lift $\bar{f}:\R\rightarrow\R$ of $f$ as 
follows.   

Restricting the Reeb-flow to the boundary of the disk slice $D_0=\{0\}\times D\subset Z_1$, gives us a smooth map $\phi:\R\times\partial D_0\rightarrow\partial Z_1$.  Reparameterising if necessary, we may assume that all points in $\partial D_0$ have first return time $1$.  Then $\phi$ restricts to a map 
$\phi:[0,1]\times\partial D_0\rightarrow\partial\T_1$ whose projection onto the 
$\partial D$ factor of $\partial Z_1=\R/\Z\times\partial D$ gives us a homotopy 
\[
 		F:[0,1]\times\partial D\rightarrow\partial D
\]
from $F(0,\cdot)=\id$ to $F(1,\cdot)=f$.  As described above, $F$ has a unique lift to a homotopy 
of $\R$ starting at the identity and ending at a lift of $f$, which we take to be $\bar{f}$.

\begin{definition}\label{D:total_rot_number_on_boundary}
For each $n\in\N$ define the \emph{total rotation number of $\psi^n$ on the boundary} to be the real number 
\[
 			\Rot_{\psi^n}(\partial D):=\Rot(\psibar_n), 
\]
where $\psibar_n:\R\rightarrow\R$ is the canonical lift, just described, of $\psi^n:\partial D\rightarrow\partial D$ induced by the flow on the boundary of $\T_n$.  
\end{definition}
It is easy to show that for each $n\in\N$, $\Rot_{\psi^n}(\partial D)=n\Rot_\psi(\partial D)$.  

\begin{definition}
Let $n\in\N$.  For $p\in\Fix(\psi^n)$ we define the \emph{infinitesimal rotation number} of $\psi^n$ at $p$ to be 
\[
 		\Rot_{\psi^n}(p):=\frac{1}{2}\lim_{k\rightarrow\infty}\frac{\mu(\gamma_p^k)}{k}
\]
where $\gamma_p:S^1\rightarrow Z_n$ is the periodic orbit passing through the disk slice $D_0$ at the 
point $p$, and $\mu(\gamma)$ denotes the Conley-Zehnder index of $\gamma$ as described in 
\ref{D:Conley-Zehnder_index}. 
\end{definition}
Again, it is easy to show that for each $k\in\N$, $\Rot_{\psi^{kn}}(p)=k\Rot_{\psi^n}(p)$.  

\begin{definition}\label{D:twist_interval}
 Let $p\in\Fix(\psi^n)$ be an interior fixed point.  Define the \emph{twist interval} of $p$, as a 
fixed point of $\psi^n$, to be the open interval of real numbers  
\[
\Twist_{\psi^n}(p):=\Big(\min\{\Rot_{\psi^n}(p),\Rot_{\psi^n}(\partial D)\}\ ,\,\max\{\Rot_{\psi^n}(\partial D),\Rot_{\psi^n}(p)\}\Big). 
\]
This interval could be empty.  
\end{definition}
We emphasize, that all three definitions above are implicitely \emph{with respect to a choice of 
data} $( Z_{1},\lambda_{1})$, that is, a Reeb-like mapping torus generating the disk map $\psi$.  

\subsection{Finite Energy Foliations for Mapping Tori} 
In our discussion here of finite energy foliations associated to Reeb-like mapping tori, as opposed to 
more general three-manifolds (see definition \ref{D:finite_energy_foliation}), the leaves will come in 
only two forms; cylinders and half cylinders.  

Consider a Reeb-like mapping torus $(Z_n,\lambda_n)$.  Let $\Jtilde_n$ be an almost complex 
structure on $\R\times Z_n$ that is compatible with $\lambda_n$ in the 
sense described in section \ref{S:contact_forms_and_holomorphic_curves}. 

\begin{remark}
 It is possible to do everything that follows under the additional assumption that $\Jtilde_n$ 
is the lift of an almost complex structure $\Jtilde_1$ on $\R\times Z_1$ that is compatible with 
$\lambda_1$.  In other words, that the lift $\Jtilde_\infty$ to $\R\times Z_\infty$ 
is invariant under pull-back by the $1$-shift automorphism, or deck transformation,  $(a,\z,\x,\y)\mapsto(a,\z+1,\x,\y)$.  This is potentially a very useful symmetry, giving us 
positivity of intersections between leaves of foliations for different iterates.  But we 
do not use this for any of the arguments in this paper.  
\end{remark}

A \emph{cylinder} will refer to a finite energy pseudoholomorphic curve of the form 
\[
 	[S^2,i,\{0,\infty\},\u], 
\]
so called because $S^2\backslash\{0,\infty\}$ can be identified with $\R\times S^1$.  In other words, 
a cylindrical leaf in a 
finite energy foliaton $\Ftilde$ associated to a Reeb-like mapping torus $(\T_n,\lambda_n)$ with 
compatible almost complex structure $\Jtilde_n$, will mean the image of an embedded solution 
$\u:\R\times S^1\rightarrow\R\times\T_n$ 
to the non-linear Cauchy-Riemann equation, having finite $E$-energy.  

A \emph{half cylinder} leaf in $\Ftilde$ will refer to a finite energy pseudoholomorphic curve (with boundary) 
of the form 
\[
 	[D,i,\{0\},\v]
\]
where $D=\{z\in\C|\ |z|\leq1\}$, because $D\backslash\{0\}$ can be identified with the half infinite cylinder  
$[0,\infty)\times S^1$.  As this is our first reference to pseudoholomorphic curves with boundary, let us be 
more precise.  Any representative of a half cylinder leaf can be identified with an embedded solution 
\[
 		\v=(b,v):[0,\infty)\times S^1\rightarrow\R\times\T_{n}
\]
to the differential equation $T\v\circ i=\Jtilde_n\circ T\v$, 
having finite energy, and satisfying the  boundary condition, that there there exists a constant 
$c\in\R$ such that 
\[
\v(\{0\}\times S^1)\subset\{c\}\times\partial\T_{n}.  
\]
The constant $c$ depends on 
the leaf.  Indeed, if $F=\v([0,\infty)\times S^1)$ is a leaf with boundary in $\{c\}\times\partial\T_{n}$, 
then translation in the $\R$-direction to $F_{c'}$ is a leaf with boundary in $\{c+c'\}\times\partial\T_{n}$.  

For the elliptic theory to work well, a pseudoholomorphic curve with boundary is typically required to 
have each boundary component lie in a prescribed surface in $\R\times M$ having suitable properties.  
For example surfaces of the form $\{\hbox{const}\}\times L$ where $\lambda$ restricts to a closed 
form on $L$, is one possibility.  This is the situation we work with here, where $L=\partial\T_{n}$.  

A finite energy foliation $\Ftilde$ associated to a Reeb-like mapping torus $(Z_{n},\lambda_{n})$, 
must have a non-empty collection of half cylinder leaves.  These determine a unique element in 
$H_1(\partial\T;\Z_{n})$.  That is, if 
$\v=(b,v):[0,\infty)\times S^1\rightarrow\R\times\T_{n}$ represents a leaf, then 
the restriction $v(0,\cdot):S^1\rightarrow\partial\T_{n}$ is a closed loop representing an homology class 
that is the same for all leaves in $\Ftilde$.  This homology class we will refer to as the \emph{boundary condition} 
of $\Ftilde$.  

We can visualize the finite energy foliations of a Reeb-like mapping torus as in Figures  
\ref{F:example_of_a_GSSS_for_a_mappingtorus} and \ref{F:foliaton_of_mapping_torus_and_cross-section}.  

\begin{figure}[hbtp]
\begin{center}
\includegraphics[scale=.47]{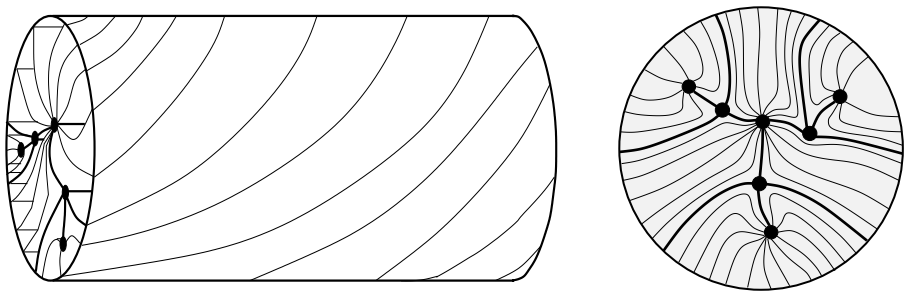} 
\end{center}
\caption{A global system of surfaces of section of a Reeb-like mapping torus.  
On the right a disk like cross-section, where the spanning 
orbits are the dots.  This can be compared with Figures \ref{Fig14} and \ref{Fig16}, although here each 
spanning orbit corresponds to a single dot rather than a pair.}
\label{F:example_of_a_GSSS_for_a_mappingtorus}  
\end{figure}    
\begin{figure}[hbt]
\begin{center}
\includegraphics[scale=.3]{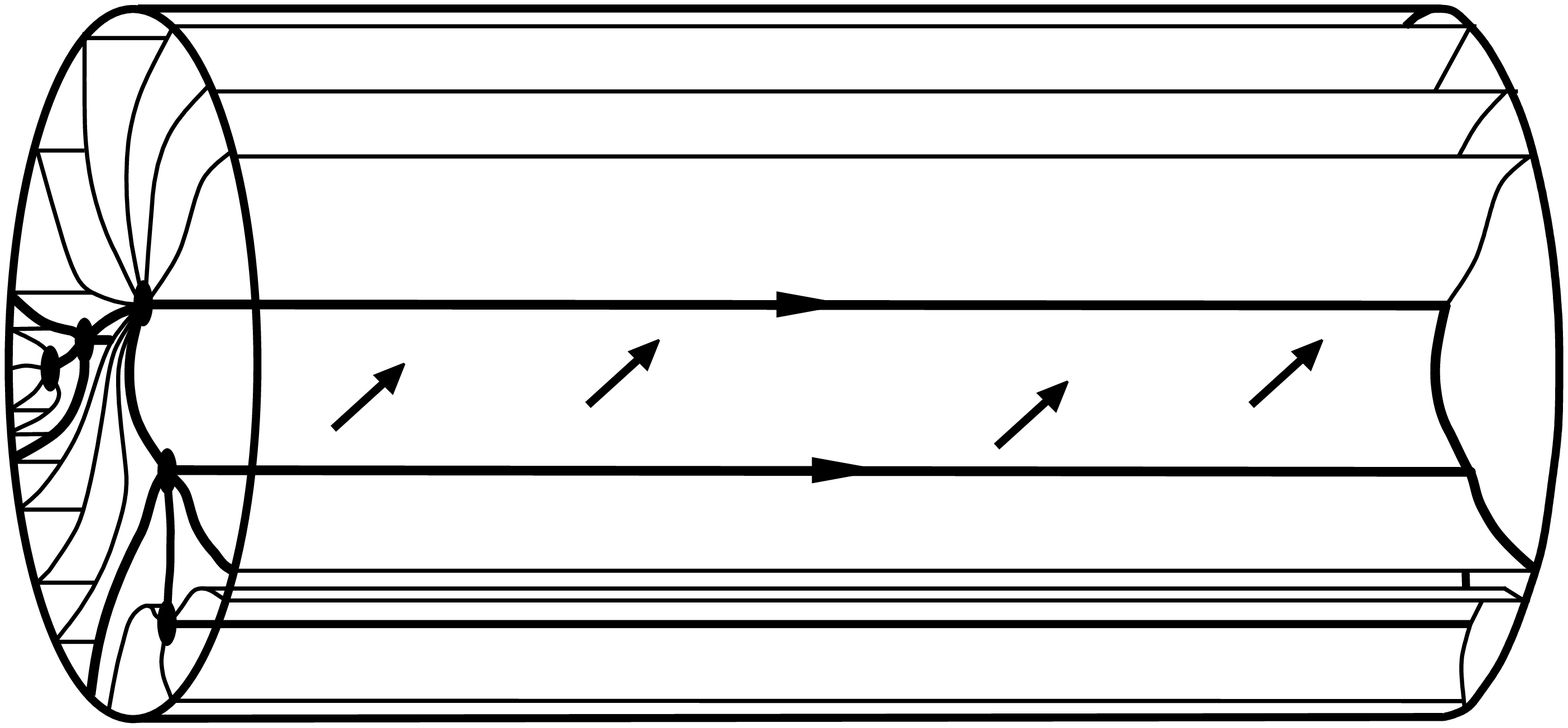}
\vspace{4mm}
\caption{An opened out view of a global system of surfaces of section of a mapping torus.  Inside, 
we see that the closure of each leaf is an embedded copy of $S^1\times[0,1]$.  The flow is transversal 
to the interior of each leaf, and tangent to those components of the boundary which lie in the 
interior of the solid torus.  These are the spanning orbits.}
\label{F:foliaton_of_mapping_torus_and_cross-section}
\end{center}
\end{figure}

The boundary condition can be described more succinctly in terms of a single integer.  

\begin{definition}\label{D:boundary_condition}
 Let $\Ftilde$ be a finite energy foliation associated to a Reeb-like mapping torus $(\T_{n},\lambda_{n})$.  
We will say that $\Ftilde$ has \emph{boundary condition $k\in\Z$} if every half cylinder leaf $F\in\Ftilde$ has boundary 
representing the homology class 
\[
		L_{n}+k[\partial D]\in H_1(\partial\T_{n};\Z)
\]
where $L_{n}$ and $[\partial D]$ are the canonical longitude and meridian introduced earlier.  
\end{definition}

A variety of boundary conditions are illustrated in figure \ref{F:different_boundary_conditions}.  

\begin{figure}[htb]
\begin{center}
\includegraphics[scale=.16]{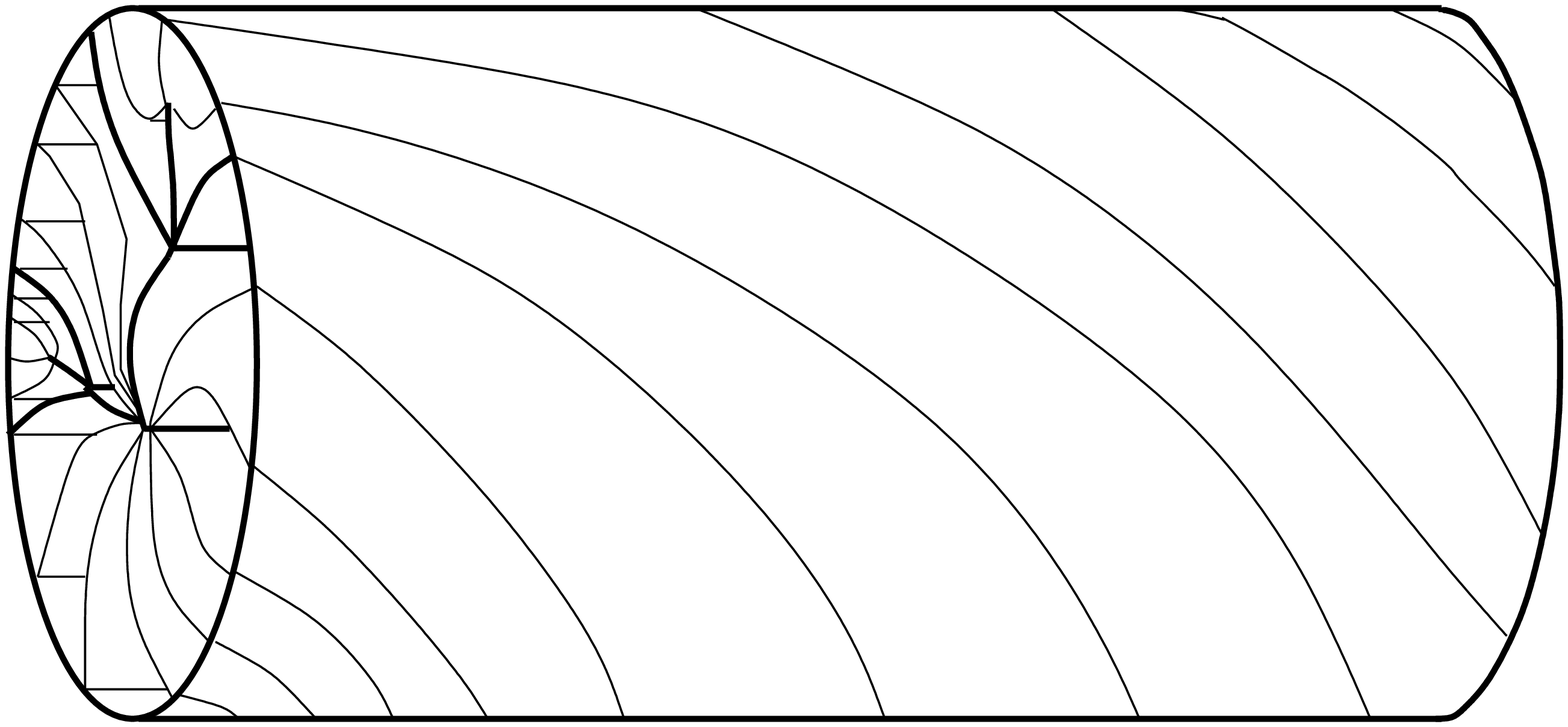}\hspace{3mm} 
\includegraphics[scale=.16]{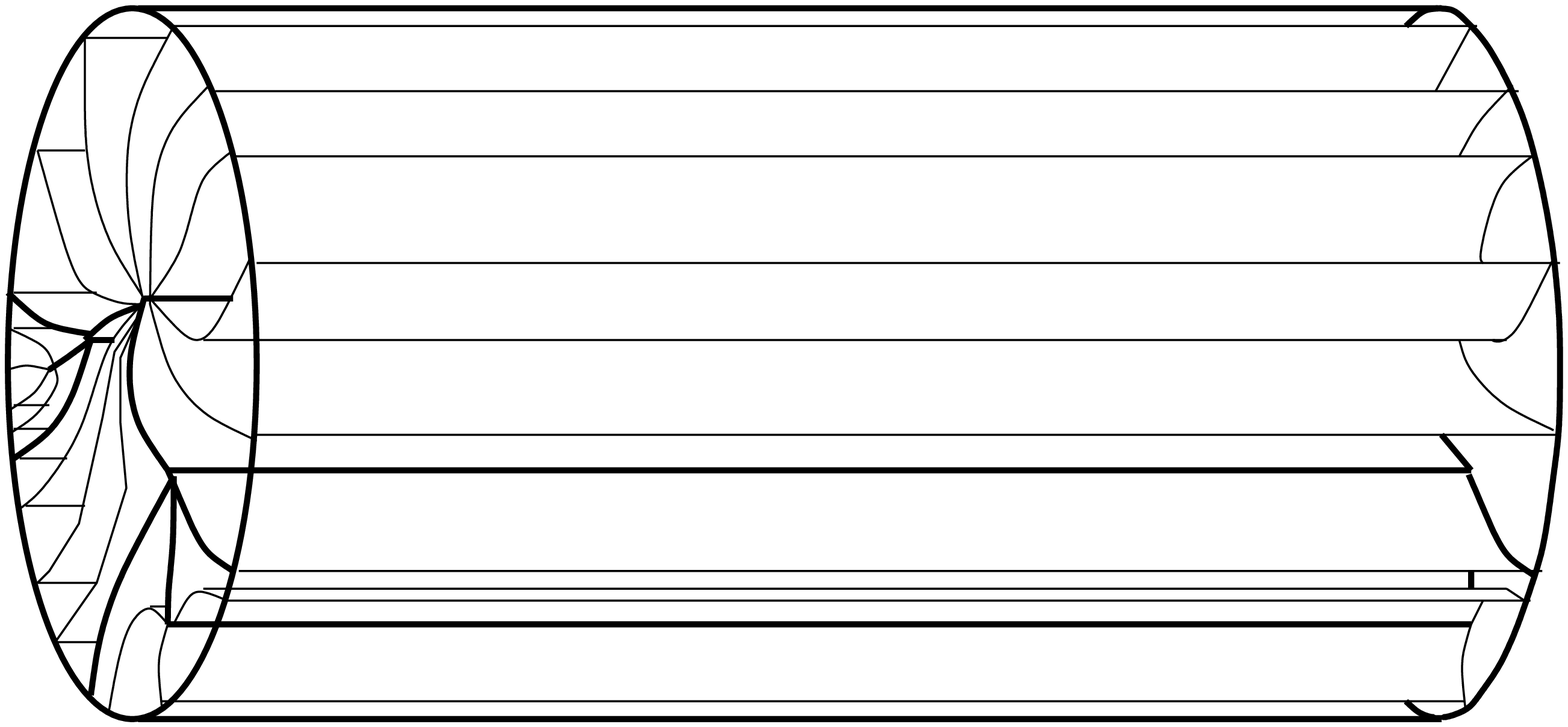}\hspace{3mm} 
\includegraphics[scale=.16]{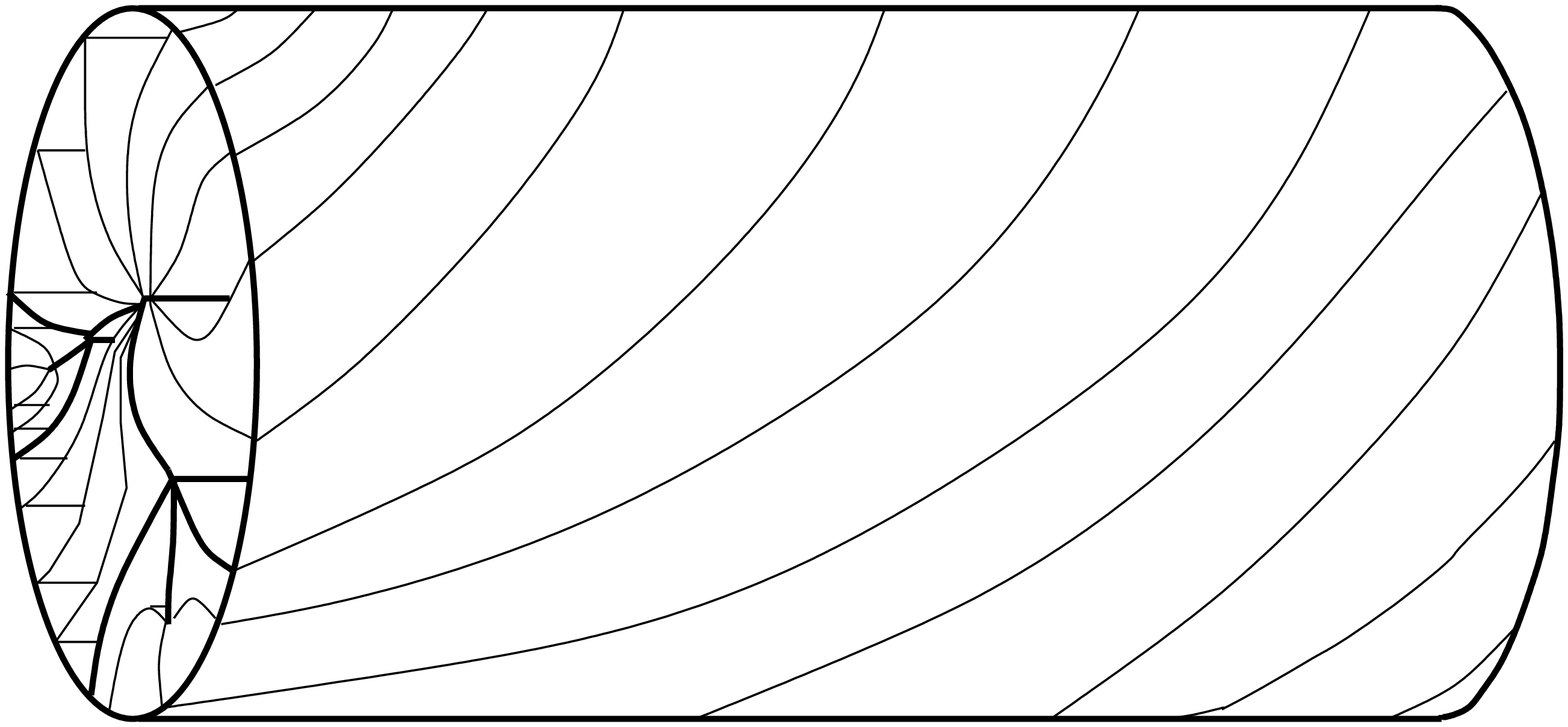} 
\end{center}
\caption{From left to right, the boundary condition is represented by the integers $-1$, $0$, and $1$ 
respectively.  The winding of the leaves is measured relative to the canonical longitude 
$L_{n}\in H_1(\partial\T_{n};\Z)$.}
\label{F:different_boundary_conditions}  
\end{figure}

\begin{remark}\label{R:leaves_intersect_disk_slice_transversally}
{The projected leaves from the finite energy foliations always intersect the disk-slice 
\[
 		D_0=\{0\}\times D\in Z_1, 
\]
transversally in the solid torus.  Hence the nice cross-sectional pictures in the figures.  
We justify this now.} 

{It is a simple matter to prescribe the almost complex structure $J_1$ on a neighborhood of 
$D_0$, without losing the necessary genericity.  This way one can arrange that there exists an 
embedded pseudoholomorphic disk $\v:D\rightarrow\R\times Z_1$ with boundary 
$\v(\partial D)=\{0\}\times\partial D_0\subset\{0\}\times Z_1\subset\R\times Z_1$, such that the 
projection of $\v$ to $Z_1$ is a parameterization of the disk slice $D_0$.  By taking 
$\R$-translates of $\v$ we foliate the whole 
hypersurface $\R\times D_0\subset \R\times Z_1$ by a $1$-parameter 
family of embedded holomorphic disks $\D$. } 

{Let $\tilde{F}\in\Ftilde$ be a leaf in a finite energy foliation.  Let $F=\pr(\tilde{F})$ 
denote the projection down to $Z_1$.  We claim that $F$ is transverse to $D_0$.   Indeed, 
$\tilde{F}$ has topological intersection number $1$ with each leaf in $\D$, 
so by positivity of intersections must also intersect each such leaf transversally in $\R\times Z_1$ 
and at a unique point.  The upshot is that $\tilde{F}$ intersects the hypersurface $\R\times D_0$ 
transversally in $\R\times Z_1$, and so the projection $F=\pr(\tilde{F})\subset Z_1$ intersects 
$D_0$ transversally in $Z_1$.  So $F\cap D_0$ is in each case a connected, compact, non-empty, 
zero or $1$-dimensional embedded submanifold of $D_0$.} 
\end{remark}

The following was proven in \cite{Bramham}. 
\begin{theorem}[Bramham]\label{T:existence_foliations_thesis}
Let $(Z_{1},\lambda_{1})$ be a Reeb-like mapping torus generating a non-degenerate disk map 
$\psi$, which coincides with an irrational rotation on the boundary of $D$.  Let $\beta\in\R$ 
denote the total rotation number of $\psi$ on the boundary, as determined by the mapping torus.  
For each $n\in\N$ let $k_{n},k_{n}+1$ be the two closest integers to $n\beta$.  
Then there exists an almost complex structure $\Jtilde_{1}$ compatible with $\lambda_{1}$, such 
that for each $n\in\N$, there exist two finite energy foliations, we will denote by 
$\Ftilde^{k_{n}}_{n}$ and $\Ftilde^{k_{n}+1}_{n}$, associated to $(Z_{n},\lambda_{n},\Jtilde_{n})$, 
where $\Jtilde_n$ is the lift of $\Jtilde_1$, which have the following 
properties:  
\begin{enumerate}
 \item \textbf{Spanning orbits:} They share a unique spanning orbit.  That is, 
 $|\P(\Ftilde^{k_n}_{n})\cap\P(\Ftilde^{k_n+1}_{n})|=1$.  The shared orbit has odd Conley-Zehnder 
index $2k_{n}+1$.  
 \item \textbf{Boundary conditions:} The boundary condition for the leaves in $\Ftilde^{j}_{n}$, 
 for $j\in\{k_{n},k_{n}+1\}$, is precisely the integer $j$. 
\end{enumerate}
\end{theorem}
The following is a basic observation. 
\begin{lemma}\label{L:linking_number_equals_boundary_condition}
Any two distinct spanning orbits $\gamma_{1},\gamma_{2}$ of a finite energy foliation $\Ftilde$ have 
linking number $\lk(\gamma_{1},\gamma_{2})=k$ where $k$ is the integer representing the boundary 
condition of $\Ftilde$.  
\end{lemma}
The next existence result is far more useful, although it already assumes the existence of a 
periodic orbit.  For the following, note that the orbit $\gamma$ having odd Conley-Zehnder index, is 
equivalent to a fixed point of the iterate $\psi^{n}$ that is either elliptic, or is hyperbolic with 
unorientable stable and unstable manifolds.  

\begin{``theorem''}[Bramham]\label{T:existence_of_foliations:_strongest_statement}
Let $(Z_{1},\lambda_{1})$ be a Reeb-like mapping torus generating a non-degenerate disk map 
$\psi$, which coincides with an irrational rotation on the boundary of $D$. 
Then there exists an almost complex structure $\Jtilde_{1}$ compatible with $\lambda_1$, such 
that for each $n\in\N$, the data $(\T_n,\lambda_n,\Jtilde_n)$ admits ``many'' finite energy foliations, 
where $\Jtilde_n$ is the lift of $\Jtilde_1$: 
Let $\gamma$ be any periodic orbit in $(Z_{n},\lambda_{n})$ homologous to the longitude $L_{n}$, and 
having odd parity Conley-Zehnder index.  Then for each $k\in\Z$, there exists a finite energy 
foliation $\Ftilde^{k}(\gamma)$ associated to $(\T_n,\lambda_n,\Jtilde_n)$ with the following two properties.   
\begin{enumerate}
 \item \textbf{Spanning orbits:} $\gamma$ is a spanning orbit.  Equivalently $\R\times\gamma$ is a leaf.    
 \item \textbf{Boundary conditions:} The boundary condition for the leaves in $\Ftilde^{k}(\gamma)$ is 
 the chosen integer $k$.  
 \end{enumerate}
\end{``theorem''}
The important novelty in this statement is that it allows to ``pre-select'' a single spanning orbit.  
Additional subtleties enter the proof when the boundary condition $k\in\Z$ 
lies in the twist interval of the pre-selected periodic orbit $\gamma$, which will appear in 
\cite{Bramham3}.  In this case it seems that the only way around certain difficulties is to use 
a refinement of contact homology developed by Momin \cite{Momin1,Momin2}.  The simpler boundary conditions will be covered in \cite{Bramham2}.  

\begin{figure}[thb]
\begin{center}
\includegraphics[trim=0 0cm 0 2.5cm, clip, scale=.45]{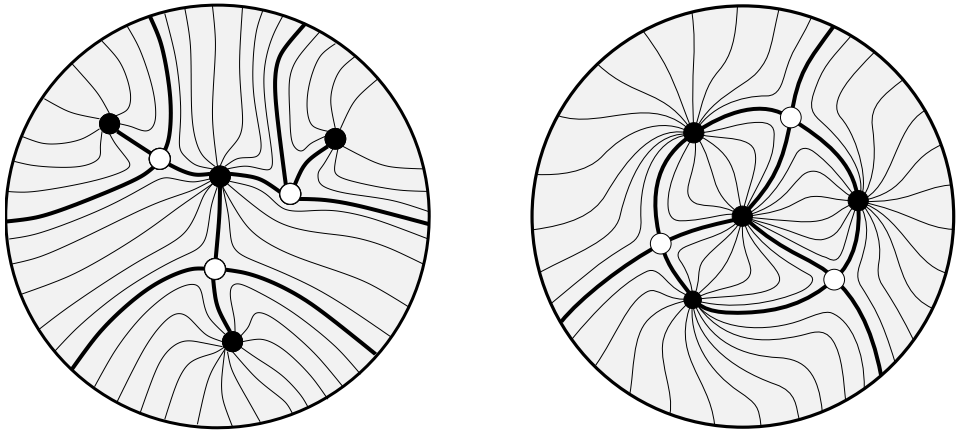}
\caption{On the left a \emph{simple} finite energy foliation.  On the right the center spanning orbit 
is not connected directly to the boundary of the mapping torus by any leaf, so this example is not simple.}
\label{F:nonsimple_foliation}
\end{center}
\end{figure}

Here is a simple way to see that theorem \ref{T:existence_of_foliations:_strongest_statement} really 
does produce many different finite energy foliations, even with the same 
boundary conditions, provided there exist enough periodic orbits.  Suppose that $\gamma_1$ and $\gamma_2$ 
are two distinct periodic orbits in $(Z_n,\lambda_n)$ that are homologous to the longitude $L_n$.  
They have a linking number $\lk(\gamma_1,\gamma_2)\in\Z$.  
Now, by Lemma \ref{L:linking_number_equals_boundary_condition}, 
for any integer $k\in\Z$ not equal to $\lk(\gamma_1,\gamma_2)$, any finite energy foliation 
associated to $(Z_n,\lambda_n)$ that contains $\gamma_1$ as a spanning orbit cannot contain $\gamma_2$, 
and vice versa.  Thus if $\Ftilde^{k}(\gamma_1)$ and $\Ftilde^{k}(\gamma_2)$ are finite energy foliations 
with boundary condition $k\in\Z$ and having $\gamma_1$ and $\gamma_2$ as spanning orbits respectively, 
then $\Ftilde^{k}(\gamma_1)\neq\Ftilde^{k}(\gamma_2)$.

The discussion of integrable disk maps in 
Section \ref{S:integrable_maps} makes it clear in certain situations what kinds of different finite 
energy foliations one can expect to find.  For example, a finite energy foliation for a given 
$(Z_n,\lambda_n)$ is in general not uniquely determined by its boundary condition and a single spanning orbit 
alone.  

For the proofs in the next section it seems useful to distinguish the following feature for finite 
energy foliations of a Reeb-like mapping torus.    
\begin{definition}
Suppose that $\Ftilde$ is a finite energy foliation associated to a Reeb-like mapping torus 
on $Z_{n}$.  We will say that $\Ftilde$ is \emph{simple} if every spanning orbit is connected 
directly to the boundary of $\T_{n}$ by a leaf in $\Ftilde$.  
\end{definition}

\noindent The following are equivalent characterizations of simple:  
\begin{itemize}
 \item The projection down to $Z_n$ of the rigid leaves intersects the disk slice $D_0$ in a 
tree-like graph.  Refer also to remark \ref{R:leaves_intersect_disk_slice_transversally}.  
 \item There are no closed ``cycles'' of leaves.  
 \item There are no Fredholm index-2 whole cylinders.  
\end{itemize}
Perhaps the simplest example of a situation where one finds a non-simple finite energy foliation 
is the following.  

Figure \ref{F:monotonetwist_map} depicts the flow lines of some autonomous smooth Hamiltonian $H:A\rightarrow\R$ on the closed annulus $A=\R/\Z\times[0,1]$.  
The time-$T$ map of the flow, for $T>0$, is an area preserving twist map $\phi$ in the sense of Poincar\'e 
and Birkhoff.  In fact for $T>0$ sufficiently small $\phi$ is a monotone twist map, meaning  
that $\phi=(\phi_1,\phi_2):A\rightarrow A$ where $\frac{\partial\phi_1}{\partial x_2}>0$ 
in coordinates $(x_1,x_2)$. 

\begin{figure}[thb]
\begin{center}
\includegraphics[scale=.25]{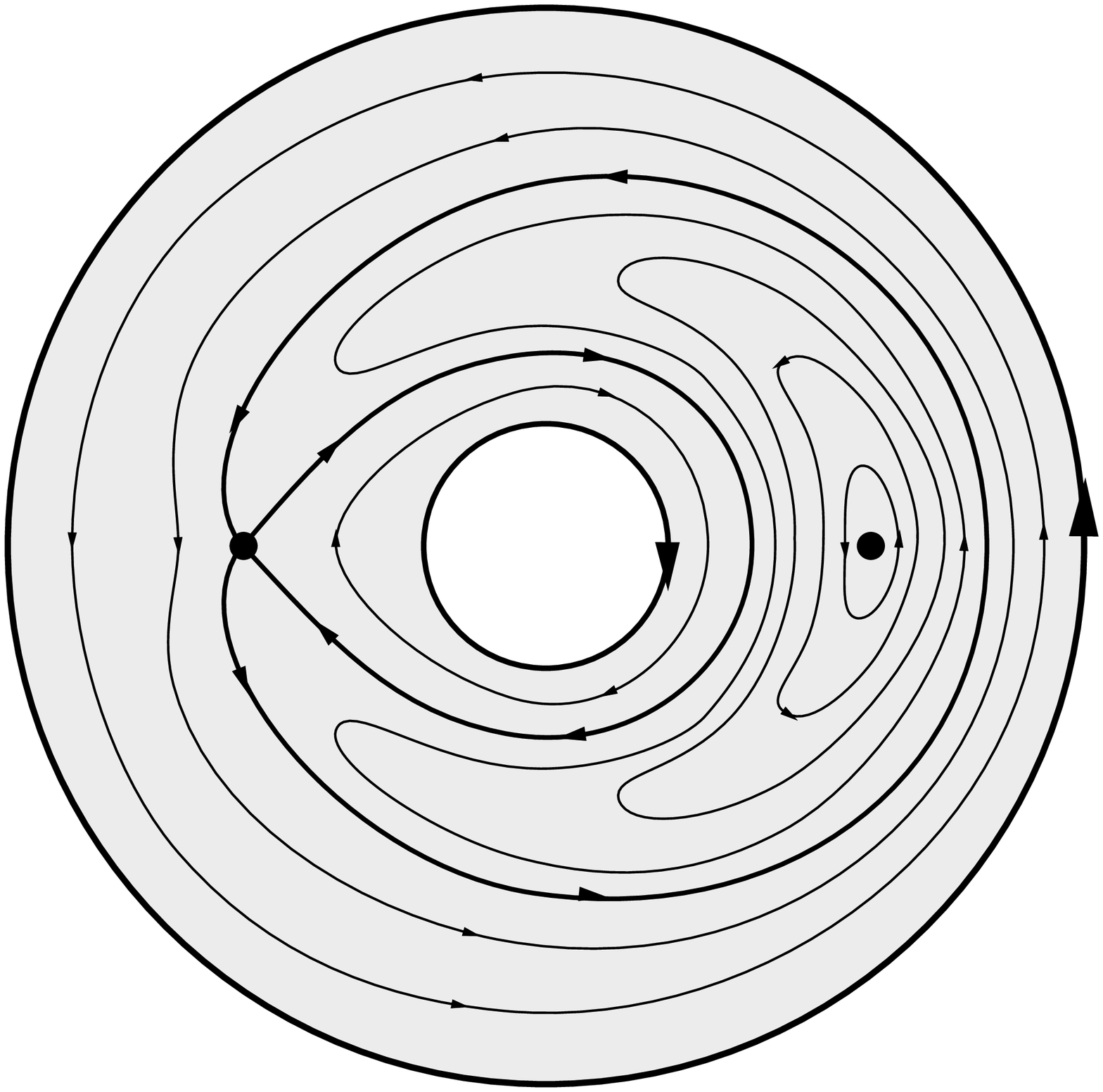}
\caption{}
\label{F:monotonetwist_map}
\end{center}
\end{figure}

Let us now consider the disk map $\psi:D\rightarrow D$ which one obtains by shrinking the inner circle to a 
point, which then corresponds to an elliptic fixed point of $\psi$ at the origin.  Let $(Z_1,\lambda_1)$ 
be a Reeb-like mapping torus generating $\psi$, chosen so that the induced total rotation number on 
the boundary lies in the interval $(0,1)$.  Let us denote by $\gamma_0$ the periodic orbit corresponding 
to the fixed point $0\in D$.  Then a finite energy foliation $\Ftilde$ associated to 
$(Z_1,\lambda_1)$ that has $\gamma_0$ as a spanning orbit, and has boundary condition $0$, would, it 
turns out, have to look as in Figure \ref{F:monotonetwist_map_and_foliation}, which is a non-simple 
foliation. 

The property of $\Ftilde$ having a chain of rigid leaves surrounding the 
spanning orbit $\gamma_0$ and the ``twisting'' property going on about the fixed point $0$ relative to 
the boundary behavior, are not coincidental.  In the following precise sense there is a twist fixed point 
if and only if there is a non-simple finite energy foliation.  

\begin{figure}[thb]
\psfrag{a}{\begin{footnotesize}$-1$\end{footnotesize}}
\psfrag{b}{\begin{footnotesize}$1$\end{footnotesize}}
\psfrag{c}{\begin{footnotesize}$0$\end{footnotesize}}
\psfrag{gamma}{$\gamma_0$}
\begin{center}
\includegraphics[scale=.43]{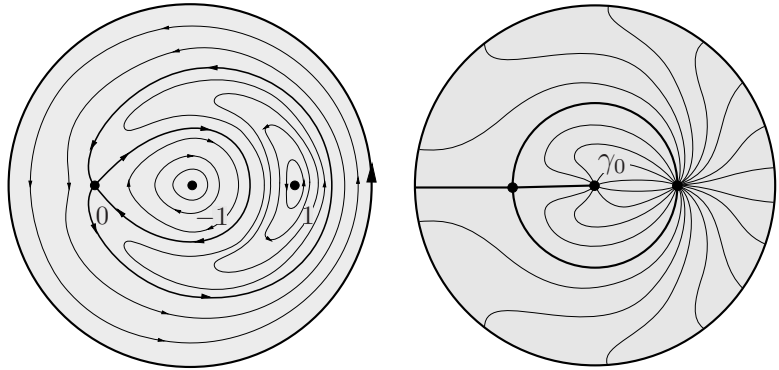}
\caption{In a simple example.  On the left the dynamics, with Conley-Zehnder indices, on the right one 
of the corresponding finite energy foliations.}
\label{F:monotonetwist_map_and_foliation}
\end{center}
\end{figure}

 \begin{lemma}\label{L:nonsimple_foliation_implies_twist_and_converse}
Suppose that $\Ftilde$ is a finite energy foliation associated to a Reeb-like mapping torus 
$(\T_n,\lambda_n)$.  Let $k\in\Z$ be the boundary condition for $\Ftilde$.  
Then:
\begin{enumerate}
 \item If there exists a spanning orbit that is not connected directly to the boundary of $\T_n$ by a 
leaf in $\Ftilde$, then there exists a spanning orbit, not necessarily the same one, having $k\in\Z$ in 
its twist interval.  
 \item Every spanning orbit in $\Ftilde$ that is connected directly to the boundary of $\T_n$ by at least one 
leaf, corresponds to a fixed point of $\psi^n$ that does not have $k\in\Z$ in its twist interval.  In particular, this implies the converse to (1).  
\end{enumerate}
\end{lemma}
\begin{proof}(Of part (2) by pictures)
Suppose that there exists a spanning orbit $\gamma$ which has $k$ in its twist interval.  
We will argue there cannot be a leaf in $\Ftilde$ which connects $\gamma$ directly to the boundary 
of $\T$. 

Due to the twist, the infinitesimal rotation number of $\gamma$ lies on the opposite side of the value $k$ 
to the rotation number describing the boundary behavior.  The behavior of the leaves and the flow 
near $\gamma$ and near the boundary are as in the figure.  

\begin{figure}[htbp]
\begin{center}
\includegraphics[scale=.25]{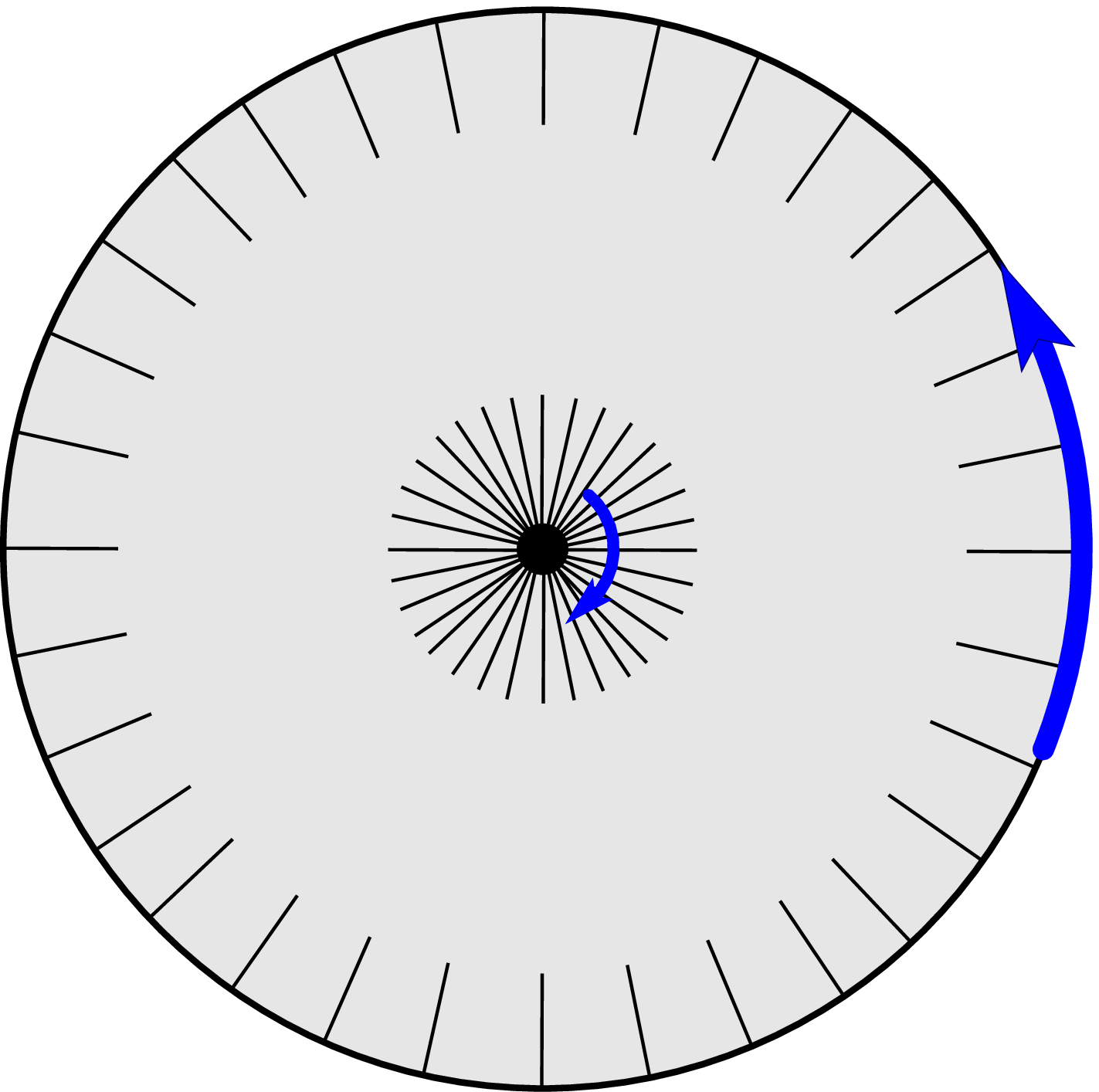}
\end{center}
\end{figure}

There is no way to complete the picture by connecting a leaf from the spanning orbit $\gamma$ to 
the boundary without contradicting the transversality of the flow to the leaf;  
near the spanning orbit the flow winds around slower than the leaves, while the flow winds faster than the leaves 
near the boundary.
\end{proof}

The example in Figure \ref{F:monotonetwist_map_and_foliation} is just a special case of the integrable maps which are discussed in generality in Section \ref{S:integrable_maps}.

\subsection{A Proof of the Poincar\'e-Birkhoff Fixed Point Theorem}
In 1913 Birkhoff in \cite{Birkhoff1} proved the following conjecture of Poincar\'e\footnote{Allegedly the effort cost 
Birkhoff 30 pounds in weight \cite{BG}.} known as the Poincar\'e-Birkhoff fixed point theorem.  See 
also \cite{Birkhoff2}.  Here $A$ denotes the closed annulus $\R/\Z\times[0,1]$ and 
$\Atilde$ the universal covering $\R\times[0,1]$ with respect to the projection map $\pi(x,y)=([x],y)$.  
\begin{theorem}[Birkhoff]\label{T:Poincare-Birkhoff_fixedpoint_thm_for_annuli}
Let $\psi:A\rightarrow A$ be an area preserving, orientation preserving, homeomorphism of the closed annulus 
with the following ``twist'' condition.  There exists a lift $\psitilde=(\psitilde_1,\psitilde_2):\Atilde\rightarrow\Atilde$ such that for all 
$x\in\R$, 
\[
 		\psitilde_1(x,0)>x \qquad\mbox{and}\qquad \psitilde_1(x,1)<x.  
\]
Then $\psi$ has at least two fixed points which also lift to fixed points of $\psitilde$.  
\end{theorem}
In our framework we consider a smooth disk map $\psi:D\rightarrow D$ represented as the first return 
map of a Reeb-like mapping torus $(\T_1,\lambda_1)$.  We observed in Section 
\ref{S:real_valued_rotation_numbers} that the choice of a mapping torus allows us to assign canonical 
twist intervals to each fixed point of $\psi$, and an integer to pairs of distinct fixed points, called 
the linking number.  Recall that the twist interval of $\psi$ at a fixed point $p$ was defined as the 
open interval of real numbers bounded by the (real valued) rotation number of $\psi$ on the boundary 
of $D$ and the (real valued) infinitesimal rotation number of $\psi$ at $p$.  

We will prove the following reformulation of Theorem \ref{T:Poincare-Birkhoff_fixedpoint_thm_for_annuli}.  

\begin{theorem}\label{T:Poincare-Birkhoff_fixedpoint_thm_for_disks}
Suppose that $\psi:D\rightarrow D$ is a $C^\infty$-smooth orientation preserving, area preserving, diffeomorphism.  Let 
$(\T_1,\lambda_1)$ be a Reeb-like mapping torus generating $\psi$.  If $\psi$ has an interior 
fixed point $p\in D$, for which there exists an integer $k\in\Z$ such that 
\[
			k\in\Twist_\psi(p), 
\]
then $\psi$ has at least two fixed points $x_1,x_2$, distinct from $p$, such that the corresponding periodic 
orbits $\gamma_1,\gamma_2:S^1\rightarrow Z_1$ have linking numbers 
$\lk(\gamma_1,\gamma_p)=\lk(\gamma_2,\gamma_p)=k$.  (In fact one also finds that  
$\lk(\gamma_1,\gamma_2)=k$.)  
\end{theorem}
The Poincar\'e-Birkhoff theorem, also in more general formulations than the statement in Theorem 
\ref{T:Poincare-Birkhoff_fixedpoint_thm_for_annuli}, has a long and beautiful history.  We mention just 
a few references \cite{Carter,ConleyZehnder,Ding,Franks3,Franks4,Guillou,K,LeCalvezWang,Neumann,Poincare}. 
It is perhaps not surprising that a variational approach in the spirit of Floer or Conley and Zehnder 
should produce a proof in the smooth category.  Indeed we recall that this 
statement led Arnol'd to make his famous conjecture.  

Nevertheless, we present now a proof by the first author of Theorem \ref{T:Poincare-Birkhoff_fixedpoint_thm_for_disks} using 
finite energy foliations, which is vaguely reminiscient of the 
simple argument that applies only to monotone twist maps.  The existence of a \emph{second} fixed 
point arises in a surprising way, and not by using indices of fixed points.  

\begin{proof}(Of Theorem \ref{T:Poincare-Birkhoff_fixedpoint_thm_for_disks}) 
We wish to apply theorem \ref{T:existence_of_foliations:_strongest_statement} which gives 
us the existence of certain finite energy foliations.  Currently to use this result requires that the behavior of the 
disk map be a rigid (irrational) rotation on the boundary circle.  So we first make an elementary 
argument to reduce the general case to this one.  The reader who wishes to skip this should jump 
to step 1.    

{We are given a $C^\infty$-smooth orientation preserving, area preserving, diffeomorphism $\psi:D\rightarrow D$ with an interior fixed point $p$, and an integer $k$ 
lying in the twist interval of $p$.  Observe that on any $\varepsilon$-neighborhood of the boundary of $D$ 
we can modify $\psi$ to obtain a new map $\psi'$ with the following properties.  $\psi'$ agrees with 
$\psi$ outside of the boundary strip, is a rigid rotation on the boundary of $D$, and is still a $C^\infty$-smooth orientation preserving, area preserving, diffeomorphism.  For example by composing $\psi$ with a suitable Hamiltonian diffeomorphism 
for a Hamiltonian that is constant outside of the $\varepsilon$-neighborhood of $\partial D$.  Moreover, 
one can do this in such a way that on the open $\varepsilon$-neighborhood $\psi'$ sends points around in one direction further, in an angular sense, than $\psi$ does, and ``speeding up'' as you approach the boundary 
circle.  Pick this direction so as to enlarge the twist interval.  If one does this for 
$\varepsilon>0$ sufficiently small, one can arrange 
that $\psi'$ has the property that any fixed points within the 
$\varepsilon$-neighborhood of the boundary have linking number with $p$ different from $k$.  To say all this 
rigorously one should of course work with a lift, on the complement of $p$, to the universal covering, but 
the idea is simple.  Thus, it suffices to prove the assertion for this modified map $\psi'$, because 
the two fixed points we find will automatically 
be fixed points of the unmodified map $\psi$.}  

{Similarly, to apply the existence theorem 
\ref{T:existence_of_foliations:_strongest_statement} it will be convenient to assume that 
the fixed point $p$ is elliptic.  A similar argument to the one just described allows to modify the 
disk map $\psi$ on an $\varepsilon$-punctured-neighborhood of the fixed point $p$, this time so as 
to send points around $p$ faster in the opposite direction, in an angular sense, increasing the twist 
interval still further.  This way we can arrange that for the new map the eigenvalues of the 
linearization $D\psi'(p)$ lie on the unit circle, so that $p$ is elliptic.  The upshot is 
that without loss of generality we can assume that our disk map is a rigid rotation on the boundary of 
the disk (with any rotation number, in particular we may take it to be irrational), and that the fixed 
point $p$ is elliptic.  For the rest of the argument we will make these assumptions on $\psi$.}  

Pick a Reeb-like mapping torus $(Z_1,\lambda_1)$ generating 
the disk map $\psi$.  Let $\gamma_p:S^1\rightarrow Z_1$ be the simply covered periodic orbit corresponding 
to the fixed point $p$.  Since $p$ is elliptic, $\gamma_p$ has odd parity Conley-Zehnder index.  

\textbf{Step 1}  The non-degenerate case:  If the disk map, equivalently the Reeb flow, is non-degenerate, 
then the existence theorem  \ref{T:existence_of_foliations:_strongest_statement} applies immediately 
and provides a finite energy foliation $\Ftilde$ associated to $(\T_1,\lambda_1)$ which has the 
odd index orbit $\gamma_p$ for a spanning orbit, and has boundary condition the integer $k$.  

By part (2) of lemma \ref{L:nonsimple_foliation_implies_twist_and_converse} $\gamma_p$ cannot 
be connected to the boundary of $\T_1$ by a leaf in $\Ftilde$ because $k\in\Twist_\psi(p)$.  
Thus $\gamma_p$ is ``enclosed'' by a chain of rigid leaves, as for example 
in figure \ref{F:foliation_for_a_twistmap}.  
\begin{figure}[htb]
\psfrag{gamma}{$\gamma_p$}
\psfrag{C1}{$C_1$}
\psfrag{C2}{$C_2$}
\begin{center}
\includegraphics[scale=.28]{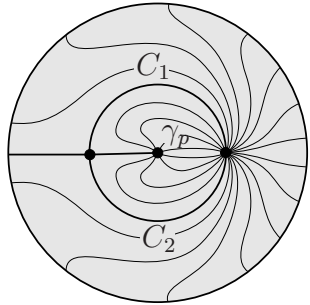}
\caption{A chain of rigid leaves $C_1$ and $C_2$ enclosing $\gamma_p$.}
\label{F:foliation_for_a_twistmap}
\end{center}
\end{figure}

Note that the chain of rigid leaves has to include at least two spanning orbits, else the enclosed 
region has the property that points move in only one direction normal to its boundary, which would 
contradict the area preserving property of $\psi$.  

Any two spanning orbits besides $\gamma_p$ have linking number $k$ with $\gamma_p$ because this is the boundary 
condition for the foliation $\Ftilde$.  Thus the theorem is proven provided $\psi$ is non-degenerate.  

\textbf{Step 2} It is not immediately obvious how to complete the proof to the degenerate case while retaining 
\emph{two} new fixed points.  This part really uses a rigidity of the holomorphic curves, and not 
merely positivity of intersections. 

Suppose $\psi$ has possibly degenerate fixed points.  Carrying out the above argument 
for a sequence of non-degenerate perturbations of $\psi$, keeping $p$ fixed, we obtain a sequence of 
finite energy foliations $\{\Ftilde_j\}_j\in\N$ having $\gamma_p$ as a spanning orbit and boundary condition $k$.  
By the argument in step 1 each $\Ftilde_j$ has at least two spanning orbits besides $\gamma_p$.  These have 
uniformly bounded period, so as $j$ goes to 
infinity they must converge to periodic orbits for the unperturbed mapping torus $(\T_1,\lambda_1)$ that 
correspond to 
fixed points for $\psi$.  The only concern is that all the spanning orbits might collapse onto a single 
limiting orbit, when the theorem requires two.  

We argue as follows.  Suppose that indeed all spanning orbits from the sequence $\Ftilde_j$, besides $\gamma_p$, 
converge to a single periodic orbit $\sigma$ in the limit as $j\rightarrow\infty$.  {Note first that the twist condition prevents $\sigma$ from coinciding with  $\gamma_p$.  Indeed, the sequence of binding orbits which converge to 
$\sigma$ have linking number $k$ with $\gamma_p$, and this prevents them entering a small neighborhood 
of $\gamma_p$ on which all points rotate either much faster or must slower than $k$ depending on the 
direction of the infinitesimal twisting at $p$.  So $\sigma\neq\gamma_p$.}

\begin{figure}[htbp]
\psfrag{P}{\begin{scriptsize}\end{scriptsize}}
\psfrag{Q}{\begin{scriptsize}\end{scriptsize}}
\psfrag{C1}{\begin{scriptsize}\end{scriptsize}}
\psfrag{C2}{\begin{scriptsize}$S_j$\end{scriptsize}}
\psfrag{Pinf}{\begin{scriptsize}$\sigma$\end{scriptsize}}
\psfrag{gamma}{\begin{footnotesize}$\gamma_p$\end{footnotesize}}
\psfrag{S}{\begin{scriptsize}\end{scriptsize}}
\begin{center}
\includegraphics[scale=.25]{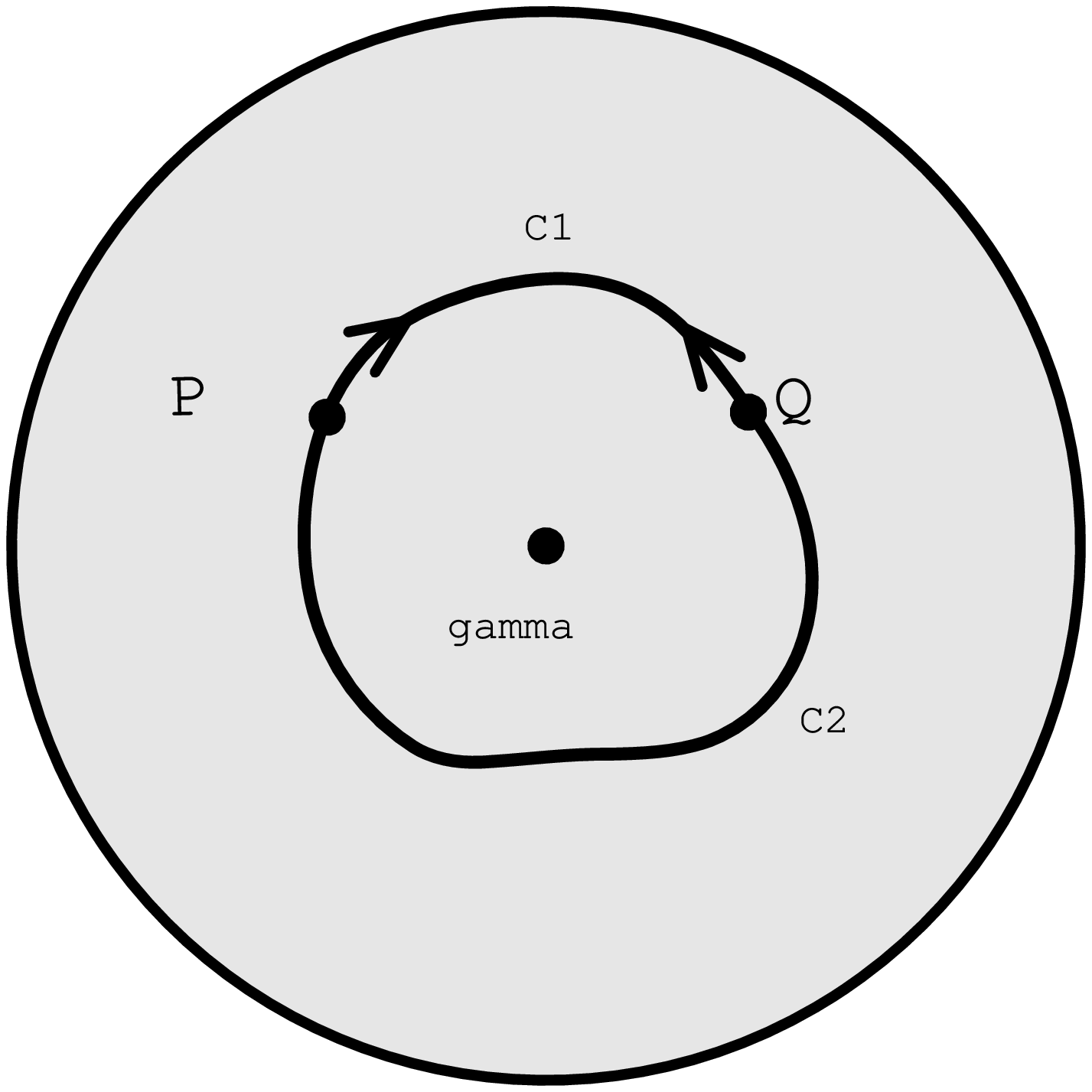}\hspace{12mm}
\includegraphics[scale=.25]{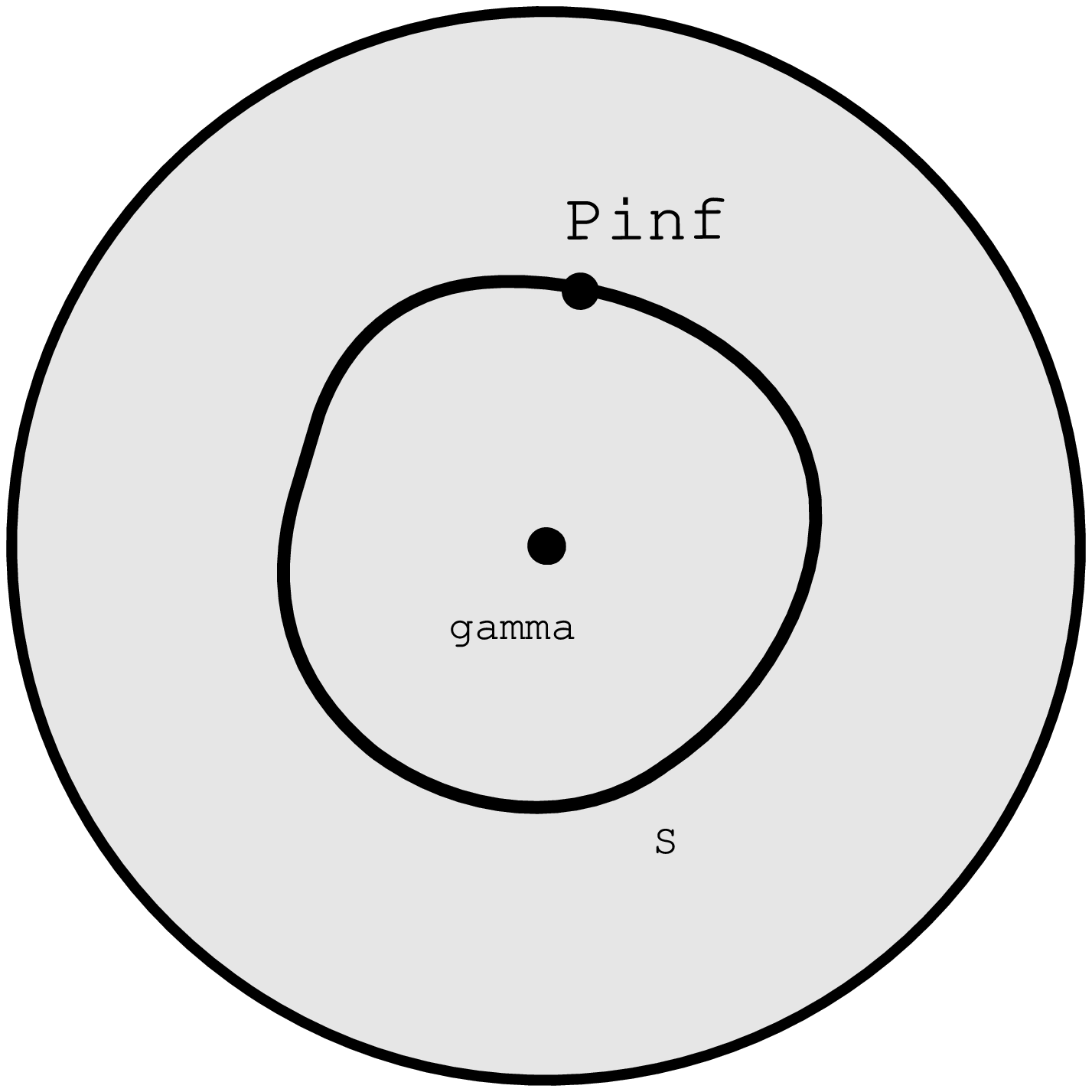}
\caption{Illustrating the worst case scenario when the only two spanning orbits, besides the twist orbit 
$\gamma_p$, collapse onto a single orbit $\sigma$ when taking a limit of a sequence of non-degenerate perturbations.}
\label{F:collapsing}
\end{center}
\end{figure}
For each $j$ there exists a closed cycle of rigid leaves in $\Ftilde_j$, surrounding $\gamma_p$.  
Let $S_j$ denote the intersection of such a closed cycle with the disk slice $D_0$.  
{Each $S_j$ is a continuously embedded closed loop in $D_0$, see remark 
\ref{R:leaves_intersect_disk_slice_transversally}.  In particular each is non-empty, compact and 
connected. Taking a subsequence, we may assume that the sequence of sets $S_j$ converges 
in the Hausdorff metric sense, to a non-empty compact set $S_\infty\subset D$.  Each $S_j$ is connected 
implies that $S_\infty$ is connected.  (Although the path connectedness need not pass to the limit.)  
Moreover, if $S_\infty$ is disjoint from $p$ and from $\partial D$, it is not hard to see that 
$p$ and $\partial D$ must both lie in different components of $D\backslash S_\infty$.  We will establish 
that $S_\infty\cap(\{p\}\cup\partial D)=\emptyset$ in a moment.} 

As $j\rightarrow\infty$ the spanning orbits corresponding to fixed points in $S_j$ converge 
to $\sigma$.  This means that all their periods converge to that of $\sigma$.  Which means, by Stokes' 
theorem, that the $d\lambda_j$-energy 
of the leaves in the $S_j$ cycles converge uniformly to zero as $j\rightarrow\infty$.  Thus if $F_j\in\Ftilde_j$ 
is any sequence of leaves for which the projection $\pr(F_j)$ meets $D_0$ in a point in $S_j$, then any convergent subsequence 
gives in the limit a pseudoholomorphic whole cylinder $C$ having zero $d\lambda$-energy, and 
finite $E$-energy.  
Such a solution $C$ is either constant or is a spanning cylinder for a periodic orbit (even though the periodic orbits 
are not necessarily non-degenerate).  From the homology class of 
$C$ we find that it is non-constant, and moreover that its spanning orbit corresponds to a fixed point of $\psi$.  

{Applying this idea to various $\R$-translates of such leaves $F_j\in S_j$, we can achieve every point 
in the set $S_\infty$ as such a limiting fixed point.  It follows that every point $q\in S_\infty$ is 
a fixed point of $\psi$ and 
satisfies the following two alternatives: either $q=p$, or if $q\neq p$ then it must have linking number $k$ with $p$.  The twist condition implies that on a small punctured neighborhood of $p$ there are no fixed 
points having this linking number $k$ with $p$, and so $S_\infty$ is disjoint from some punctured 
neighborhood of $p$.  Moreover, $S_\infty$ contains at least one point distinct from $p$ because it 
contains the periodic orbit $\sigma$.  It follows, since $S_\infty$ is connected, that it is also disjoint from $p$.  Similar arguments show 
that $S_\infty$ is disjoint from the boundary of $D$.}  

Thus, as mentioned above, since $p$ and $\partial D$ are disjoint from $S_\infty$, they lie in 
different components of the complement $D\backslash S_\infty$.  The set $S_\infty$ must therefore 
have infinitely many points.  In particular at least two points.  {These are fixed points 
with the desired linking number $k$.}  
\end{proof}

\subsection{A Proof of a Theorem of Franks and Handel}\label{S:proof_of_a_theorem_of_Franks_and_Handel} 
We could summarize the last section by saying loosely that if an area preserving disk map has 
``twisting'' in a suitable sense, then one can find periodic orbits and distinguish them by topological 
means, namely by their linking numbers.  

To make this precise, let us say that a diffeomorphism $\psi:D\rightarrow D$ has ``twisting'' if, after 
making a choice that allows to define real valued rotation numbers, for example an isotopy to the identity, 
or a generating mapping torus, there exists an interior fixed point $z\in D$ such that 
\begin{equation}\label{D:twisting}
 		\Rot_{\psi}(z)\neq \Rot_\psi(\partial D),
\end{equation}
where the left hand side is the total infinitesimal rotation number of $\psi$ at $z$, and the right hand 
side is the total rotation number of $\psi$ on the boundary of $D$.  One way to define these is described  
in Section \ref{S:real_valued_rotation_numbers}, both real valued numbers.  

An equivalent way of saying this is that the twist interval of the fixed point $z$ is non-empty.  The 
twist interval grows linearly with iterates, so that for a high enough iterate it contains an integer, 
\[
 			\Twist_{\psi^n}(z)\cap\Z\neq\emptyset.  
\] 
In this situation the Poincar\'e-Birkhoff fixed point theorem, as stated in Theorem 
\ref{T:Poincare-Birkhoff_fixedpoint_thm_for_disks}, applies to $\psi^n$ and we find that for 
each integer $k\in\Twist_{\psi^n}(z)$, there exists a fixed point $x\in\Fix(\psi^n)$ (actually two fixed points) 
which has linking number $k$ with $z$, $\lk_{\psi^n}(x,z)=k$.  

It is a simple matter to work with an unbounded sequence of iterates $\psi^{n_j}$, and pick integers $k_j$ in the 
twist interval $\Twist_{\psi^{n_j}}(z)$ to be primes, and so obtain a sequence of fixed points $x_j$ for which the 
average linking numbers with $z$, 
\[
 			\frac{\lk_{\psi^{n_j}}(x_j,z)}{n_j}=\frac{k_j}{n_j}, 
\]
are already in lowest form as fractions and are therefore pairwise distinct.  Thus the sequence $(x_j)_j\in\N$ 
is a sequence of periodic points of $\psi$, lying on pairwise distinct orbits.  Moreover, $x_j$ has minimal 
period $n_j$, with $n_j\rightarrow\infty$ as $j\rightarrow\infty$.  Thus, a disk map having a twist, or some 
iterate with a twist, has infinitely many periodic orbits, and moreover periodic orbits of unbounded minimal period.  

In \cite{Neumann} the argument is pushed as far as possible.  If $\psi$ has a fixed point with non-empty 
twist interval $(a,b)$, $a,b\in\R$, then for each $N\in\N$, the number $\mu_\psi(N)$ of periodic orbits 
of $\psi$ having minimal period less than or equal to $N$ grows like $|b-a|N^2$.  Indeed, 
for each rational number $p/q\in\Q$ in $(a,b)$ the argument above gives us a periodic point of $\psi$ 
having average linking number $p/q$.  So $\mu_\psi(N)$ is at least the number of rationals in $(a,b)$ 
having denominator at most $N$ when written in lowest form.  This is equivalent to counting lattice 
points in a triangle, having relatively prime coordinates, 
\[
 	\mu_\psi(N)\geq|\{\, (p,q)\in\Z\times\Z\ |\, Na<p<Nb,\ 1\leq q\leq N\mbox{ and }\gcd(p,q)=1\}|.  
\]
A little number theory, see \cite{Neumann,Hingston}, tells us precisely what the limit of the quotient of 
the right hand side with $N^2$ is, leading to the following asymptotic estimate from below
\begin{equation}\label{E:growth_rate_for_twistmaps}
 			\lim_{N\rightarrow\infty}\frac{\mu_\psi(N)}{N^2}\geq\frac{3|b-a|}{\pi^2}.  
\end{equation}
To summarize: a disk map with a twist has at least quadratic growth of 
periodic orbits.  Although one might generically expect a far higher growth rate, quadratic is the best possible for general twist maps, see examples in \cite{Neumann}.  

Birkhoff himself applied these ideas to the problem of finding infinitely many closed geodesics on $2$-spheres.  
In many cases he showed that the dynamics can be related to that of an area preserving diffeomorphism of 
an annulus, possibly without boundary.  This holds for example whenever the metric has everywhere positive 
curvature.  When the so called Birkhoff map has a twist he could apply his fixed point theorem as above to 
complete the argument.  

But in the absence of a twist the problem of detecting infinitely many periodic orbits becomes much subtler.  
Indeed, a twist allows to conclude periodic orbits of unbounded period and to detect them by topological 
means.  But this clearly does not work for example on the identity map of the disk which has infinitely many 
periodic orbits but which are indistinguishable by linking numbers.  It was an open problem for 
some time, whether every smooth Riemannian metric on $S^2$ admits infinitely many distinct (prime) closed geodesics.  

In 1992 John Franks proved the following celebrated result, \cite{Franks} originally stated for annulus maps.  
\begin{theorem}[Franks]\label{T:Franks_theorem}
Let $\psi$ be any area preserving, orientation preserving, homeomorphism of the open or closed unit disk in 
the plane.  If $\psi$ has $2$ fixed points then it has infinitely many interior periodic orbits.  
\end{theorem}
A remarkable aspect of this statement is that no twisting type of condition is required.  Let us point 
out that by imposing seemingly mild extra assumptions one can unwittingly introduce a twist into the 
system, thereby allowing much simpler arguments which miss the whole point of this theorem.  Two examples 
which fall into this category, the first of which is a familiar one for symplectic geometers, are 
the assumption of non-degeneracy of periodic orbits, or that the rotation number on the boundary of 
the disk is irrational.  

Franks' theorem completed the proof of the conjecture regarding closed geodesics on the two-sphere, in the 
cases where Birkhoff's annulus map is well defined.  At the same time Bangert found a clever way to handle 
the cases where the Birkhoff map is not defined, \cite{Bangert}.  Let us also note that shortly thereafter 
Hingston found another route to the closed geodesics question, using equivariant Morse homology \cite{Hingston}, 
bypassing the part of the argument that required Theorem \ref{T:Franks_theorem}.  

An interesting strengthening of Franks' theorem is the following result due to Franks and Handel 
\cite{Franks2} in 2003.  In fact much more generally, they proved analogous statements 
for Hamiltonian diffeomorphisms on general closed oriented surfaces.  All of these results were later extended 
to the $C^0$-case by Le Calvez \cite{LeCalvez4}. 

\begin{theorem}[Franks-Handel]\label{T:Franks-Handel}
For a $C^\infty$-smooth Hamiltonian diffeomorphism $\psi:S^2\rightarrow S^2$ 
having at least three fixed points, either there is no bound on the minimal period of its periodic orbits, 
or $\psi$ is the identity map.  
\end{theorem} 

In the case when the map is not the identity, they obtain the following lower bound on the growth rate 
of periodic orbits.  There exists a constant $c>0$, depending on $\psi$, such that 
the number of periodic orbits $\mu_\psi(N)$ with period less than or equal to $N\in\N$ satisfies
\[
 		\mu_\psi(N)\geq cN  
\]
for all $N\in\N$.
A significant improvement on this growth rate estimate follows from work by Le Calvez \cite{Calvez}.  The 
main result in \cite{Calvez} can be combined with the Franks-Handel theorem or \cite{LeCalvez4}, to conclude at least quadratic growth of periodic orbits.  Namely, that there exists $c>0$ such that $\mu_\psi(N)\geq cN^2$ 
for all $N\in\N$.  

Since this is exactly the growth rate well known for twist maps, described in (\ref{E:growth_rate_for_twistmaps}) above, it raises the following question: if 
$\psi$ is a disk map satisfying the criteria of the Franks-Handel theorem, and $\psi$ is not the identity, 
is some iterate of it actually a twist map - that is, has a fixed point with non-empty twist 
interval.  As far as the authors are aware, the following is new.   
 
\begin{``theorem''}[Bramham]\label{T:two_porbits_implies_id_or_twist}
Let $\psi:D\rightarrow D$ be a $C^\infty$-smooth, area preserving, orientation preserving, 
diffeomorphism having at least two fixed points, and coinciding with a rigid rotation on the boundary of 
the disk.  Then either $\psi$ is the identity map, or there exists $n\in\N$ for which $\psi^n$ is a twist 
map in the following sense.  {There exists an interior fixed point $p\in D$ of $\psi^n$, having 
non-empty twist interval.}   
\end{``theorem''}
The quotation marks are because it relies on ``Theorem'' \ref{T:existence_of_foliations:_strongest_statement}.  
Presumably the condition of being a rotation on the boundary is unnecessary.  But it is not yet 
clear how to remove it.  

Notice that this statement combined with the Poincar\'e-Birkhoff fixed point theorem gives us the analogous 
result of Franks and Handel, Theorem \ref{T:Franks-Handel}, but for disk maps instead of maps on $S^2$.  By the 
discussion above it also implies the sharp quadratic growth lower bound on periodic orbits.  Moreover, it 
allows to partition the set of area preserving diffeomorphisms of the disk into the following three disjoint subsets:  
\begin{enumerate}
 \item \textbf{Pseudo-rotations:} those maps with a single periodic orbit.  
 \item \textbf{Roots of unity:} those maps for which some iterate is the identity.  
 \item \textbf{Twist maps:} in the sense that some iterate has a fixed point with non-empty twist interval.   
\end{enumerate}
The following proof will appear in \cite{Bramham1}.  
\begin{proof}(Of \ref{T:two_porbits_implies_id_or_twist})
Suppose that $\psi$ has two fixed points, and that for all $n\in\N$, $\psi^n$ has no twisting.  
In other words, every interior fixed point, of every iterate, has empty twist interval.  Then we will show 
that $\psi(z)=z$ for all $z\in D$, by showing that the vertical foliation 
(definition \ref{D:vertical_foliation}) is a finite energy foliation.  

Pick a suitable Reeb-like mapping torus $(\T_1,\lambda_1)$ generating $\psi$.  This gives us for each 
$n\in\N$ a mapping torus $(\T_n,\lambda_n)$ generating the iterate $\psi^n$.  With respect to 
these choices we associate real valued rotation numbers to $\psi$ and its iterates at each fixed 
point and to the boundary of $D$.  We proceed in three steps. 

\textbf{Step 1:} The (real valued) rotation number of $\psi$ on the boundary is an integer.  

Suppose not.  Then we may assume that $\psi$ has only non-degenerate fixed points, as such a fixed point has 
integer infinitesimal rotation number, and then the condition of no twisting would imply that the rotation 
number on the boundary is also an integer and we would be done.   

Using that $\psi$ has at least two fixed points one can show, using the foliations for 
example, or even just the Lefschetz fixed point formula, that there must be a hyperbolic fixed 
point.  Since a hyperbolic fixed point has integer infinitesimal rotation number the lack of 
twisting implies that the rotation number on the boundary is an integer.  

Let us say that a fixed point is \emph{odd} if the eigenvalues of the linearization lie either 
on the unit circle or the negative real line.  Equivalently, the corresponding periodic Reeb orbit 
is either degenerate or is non-degenerate and has odd parity Conley-Zehnder index.  
 
\textbf{Step 2:} In this step we show that all odd fixed points have the same $\lambda_1$-action,   
where we define the $\lambda_1$-action of a fixed point $x\in D$ to be the value 
\[
 			\A(x)=\int_{S^1}\gamma^*\lambda_1
\]
where $\gamma:S^1\rightarrow Z_1$ is the periodic orbit corresponding to $x$.  

Let $k\in\Z$ be the total rotation number of $\psi$ on the boundary.  
Being an integer $\psi$ has at least one fixed point on the boundary of $D$.  Let $a_0>0$ denote the 
$\lambda_1$-action of any of these boundary fixed points.  
It suffices to show that every interior odd fixed point of $\psi$ has $\lambda_1$-action 
equal to $a_0$.  

Fix any interior odd fixed point $x$.  Let $\varepsilon>0$.  We will show that
\[
 				|\A(x)-a_0|<\varepsilon. 
\]
Let $\gamma_x:S^1\rightarrow\T_1$ be the corresponding closed Reeb orbit.  
Pick $n\in\N$ large enough that $\hbox{Area}_{d\lambda_1}(D)/n<\varepsilon$.  
Take a sequence of non-degenerate perturbations $\lambda^j_1$ of $\lambda_1$ for which $\gamma_x$ is 
a closed Reeb orbit for each $\lambda^j_1$ also, with action converging to $\A(x)$.  Since $x$ is an 
odd fixed point of $\psi$, we may do this in such a way that $\gamma_x$ has odd Conley-Zehnder index 
with respect to $\lambda^j_1$.  Assume that we also 
perturbed near the boundary of $\T_1$ so that the first return map $\psi_j$ is a rotation on the boundary of $D$ 
with rotation number $k+\delta_j$, where for each $j$, $\delta_j\in(0,1)$ is irrational.  

By assumption $\psi^n$ has no twist fixed points.  Thus every twist interval has length zero, and   
$\Rot_{\psi^n}(p)=\Rot_{\psi^n}(\partial D)=nk$ for all $p\in\Fix(\psi^n)$.  
Thus for sufficiently small perturbations, that is for $j$ sufficiently large, we may assume 
that for each fixed point $p$ of $\psi_j^n$, its twist interval is in a neighborhood of $kn$.  For 
example we may assume that  
\[
 			\Twist_{\psi_j^n}(p)\subset \left(kn-\half,kn+\half\right).  
\]
Recall, by part (2) of Lemma \ref{L:nonsimple_foliation_implies_twist_and_converse}, that any finite 
energy foliation with boundary condition outside of the twist interval of each of its spanning orbits 
is simple.  Using ``Theorem'' \ref{T:existence_of_foliations:_strongest_statement} there exists a 
finite energy foliation $\Ftilde_j$ associated to $\lambda_j$ on $Z_n$, having $\gamma_x^n$ as a 
spanning orbit and any boundary condition we choose.  Let us take the boundary condition to be $kn+1$.  As 
this is outside of every twist interval $\Ftilde_j$ must be simple.  

It follows that there exists a half cylinder leaf in $\Ftilde_j$ which connects $\gamma_x^n$ to the 
boundary.  Positivity of the $d\lambda_j$-energy of this leaf and Stokes theorem give us the following 
estimate:  
\[
 	\int_{S^1}(\gamma_j^n)^*\lambda_j<\int_{L_n}\lambda_j + (nk+1)\int_{\partial D}\lambda_j
\]
where $L_n$ is any representative of the canonical longitude on $\partial\T_n$.  Dividing by 
$n$ and letting $j\rightarrow\infty$ leads to 
\[
 		\A(x)\leq a_0 + \frac{1}{n}\hbox{Area}_{d\lambda_1}(D)< a_0 + \varepsilon.  
\]
To complete the argument we need the lower bound $\A(x)>a_0 - \varepsilon$ which can be obtained in the 
same way by changing the boundary conditions to $nk-1$.

\textbf{Step 3:} In this step we show that the vertical foliation $\Ftilde^\nu(\T_1,\lambda_1)$ 
is also a finite energy foliation!  

See definition \ref{D:vertical_foliation} to recall how we define a vertical foliation.  
Take any sequence of approximating data 
$\lambda^j_1\rightarrow\lambda_1$ with the only requirement that each be non-degenerate 
(for fixed points at least) and have irrational rotation number on the 
boundary.  

Let $\Ftilde_j$ be a finite energy foliation associated to $(\T_1,\lambda^j_1)$ having boundary condition 
$k$.  Recall that $k\in\Z$ is the rotation number of $\psi$ on the boundary.  There is no need to select a spanning orbit.  Then we claim that the $d\lambda$-energies of all leaves 
converge to zero, that is $E_j\rightarrow 0$ where 
\[
 		E_j:=\sup\left\{\,\int_F d\lambda_1^j\, |\, F\in\Ftilde_j\,\right\}.
\]
Roughly this is via the following argument:  The value $E_j$ must be achieved by a leaf in $\Ftilde_j$ 
having Fredholm index $2$, since all other leaves have index $0$ or $1$ and lie on the boundary of a 
family of index-$2$ leaves, and the $d\lambda_1^j$-integral is lower semi-continuous under 
$C_{\textup{loc}}^\infty$-convergence.  Any leaf having Fredholm index $2$ has only odd index asymptotic orbits 
(or orbit).  Any sequence of such leaves has $d\lambda_1^j$-energy decaying to zero because in the limit 
the asymptotic orbits converge to odd periodic orbits of $\lambda_1$, which by step (2) all have the same action.  

It follows that the sequence $\Ftilde_j$, which has uniformly bounded $E$-energy, converges to the 
vertical foliation $\Ftilde^\nu(\T_1,\lambda_1)$ (one can also show that the limiting leaves are not constants).  Thus the vertical foliation is a finite energy foliation, which means that every 
leaf is a cylinder over a periodic orbit.  Thus the mapping torus $(Z_1,\lambda_1)$ is 
foliated by periodic orbits.  Moreover, each cylinder must be over a periodic 
orbit homologous to the longitude in $\T_1$, thus each periodic orbit represents a fixed point of $\psi$.  
\end{proof}
 
\section{Asymptotic Foliations for Disk Maps}\label{S:asymptotic_foliations}
In this final section we begin exploring the following broad question.  What is the asymptotic 
behavior of the foliations associated to arbitrarily high iterates of a disk map?  Do the leaves converge in 
a useful sense, and if so, in what way does the limit reflect anything interesting dynamically?  
In contrast, the applications in Section \ref{S:holomorphic_curves_and_disk_maps} only used the framework for  arbitrarily high, but finite numbers of iterates.  

To develop the observations we present in this section presumably requires a serious study of what one could describe as ``locally-finite'' energy foliations.   The context is as follows.  

Let $(Z_1,\lambda_1)$ be a Reeb-like mapping torus, and $J_1$ a compatible almost complex structure 
on $\R\times Z_1$.  Let $ Z_{\infty}=\R\times D$ denote the universal covering space of $Z_1$.  The contact 
form and almost complex structure lift to 
$\lambda_{\infty}$ and $J_\infty$ respectively, which are invariant under the $1$-shift automorphism 
\begin{equation}\label{D:deck_transformation}
 \begin{aligned}
  	\tau: Z_{\infty}&\rightarrow Z_{\infty} \\
		(\z,\x,\y)&\mapsto (\z+1,\x,\y). 
 \end{aligned}
\end{equation}
One would like to take a sequence of finite energy foliations $\Ftilde_n$ associated to 
data $(Z_n,\lambda_n,J_n)$, as $n$ runs over the natural numbers or some subsequence thereof, and 
extract a limiting object $\Ftilde_\infty$ associated to $(Z_\infty,\lambda_{\infty},J_\infty)$.  
One would like to arrange this so that $\Ftilde_\infty$ is a foliation by leaves that are the images of properly embedded pseudoholomorphic curves, where each leaf $F\in\Ftilde_\infty$ has a parameterization 
that is the $C^\infty_{\textup{loc}}$-limit of a sequence of parameterizations of leaves $F_n\in\Ftilde_n$.  
The foliation $\Ftilde_\infty$ would be invariant under the $\R$-action 
$c\cdot(a,m)\mapsto(a+c,m)$ on $\R\times Z_\infty$, but need not be invariant under the $\Z$-action 
generated by the transformation $\id\times\tau$, or even a finite iterate $\id\times\tau^q$.   

When the contact manifold is compact it was discovered in \cite{H1} that 
finiteness of the $E$-energy picks out those holomorphic curves which 
do not behave too wildly, see Section \ref{S:notions_of_energy}.  In the non-compact case 
$( Z_{\infty},\lambda_{\infty},J_\infty)$ there cannot be any non-constant curves with finite $E$-energy, since any such curve would pick out a periodic orbit and 
there are no periodic orbits in $( Z_{\infty},\lambda_{\infty})$.  So presumably another, weaker, condition than  
global finiteness of the $E$-energy is required, perhaps some kind of averaging version of the $E$-energy.

Ignoring these crucial technicalities for the moment, we now outline two situations where nevertheless 
it is possible to see what happens asymptotically.   
In the first situation we will describe what happens when the disk map is integrable, the second is when the disk map has only a single periodic orbit, so called irrational pseudo-rotations which seem to have generated renewed interest in recent years \cite{FayadKrikorian,FayadKatok,FayadSaprykina,Kwapisz,BCLP1,BCLP2}.  

\subsection{Detecting Invariant Circles}\label{S:integrable_maps}
{KAM theory guarantees the existence of closed invariant curves near a generic elliptic 
fixed point, or near to a given invariant circle with certain rotation number and torsion 
(infinitesimal twisting) conditions.  In other words this is a, very successful, perturbation theory.  
Computer simulations indicate that generically, or often, invariant circles and quasi-periodic behavior 
should exist in a ``global'' sense, whatever this means precisely.  An interesting source of pictures, 
numerical observations, and many questions, is MacKay's book \cite{MacKay}.  More evidence for these 
global features arises in Aubry-Mather theory, which finds quasi-periodic behavior and invariant circles 
for monotone twist maps of the annulus.  What about without the monotone twist assumption?  Note that if 
one doesn't view the annulus as a cotangent bundle then the monotone twist condition is not symplectic.   
An argument against asking this question is of course that the so called family of standard maps
$f_\tau:\R/\Z\times\R\rightarrow\R/\Z\times\R$, for parameters $\tau\in\R_{>0}$, 
\[
 		f_\tau(x,y)=(x+y+\tau\sin(2\pi x),y+\tau\sin(2\pi x)),
\]
are all monotone twists, and still there are huge open questions regarding these very explicit examples.}  

{In this section we do not claim any new results, or recover any known ones.  We merely describe an  
intriguing mechanism by which finite energy foliations asymptotically pick out all the invariant circles of 
an integrable disk map.  The circles come filtered through their rotation numbers and the elliptic 
periodic orbits they enclose. By integrable, we mean the time-$1$ map of an autonomous Hamiltonian 
$H:D\rightarrow\R$.  Of course, since we are in $2$-dimensions, integrable Hamiltonian systems 
are easily understood without anything so technical as holomorphic curves.  But this mechanism uses 
the integrability in 
a very weak sense that needs to be better understood, and holomorphic curves have little regard 
for local versus global issues.  One might hope therefore that more sophisticated variations of 
this approach work as well in very general situations.}  

Let $H:D\rightarrow \R$ be a smooth function, constant on the boundary of $D$.  Viewing the area form 
$\omega_0:=dx\wedge dy$ as a symplectic form, we have an induced, autonomous, Hamiltonian vector field $X_H$ 
uniquely solving 
\[
 		-dH(x,y)=\omega_0(X_H(x,y),\cdot)  
\]
for all $(x,y)\in D$, and automatically $X_H$ will vanish on the boundary of $D$.  Let 
$\varphi_H^t\in\Diff_{\omega_0}(D)$ denote the induced $1$-parameter family of diffeomorphisms defined 
for each $t\in\R$, and let 
\[
 		\psi_H:D\rightarrow D
\]
denote the time-$1$ map $\varphi_H^1$.  Then $\psi_H$ is an area preserving and orientation preserving  
diffeomorphism, and its long term behavior is easily understood because its orbits remain in level curves 
of $H$.  

Let us imagine that we have picked an $H$ for which the level curves are as in Figure \ref{F:integrable_diskmap}, 
and suppose that $\lambda^H_1$ is a contact form on $\T_1=\R/\Z\times D$ giving us a Reeb-like mapping 
torus for $\psi_H$.  

Let us give names to some of the features in the figure.  There are seven isolated fixed points of $\psi_H$, 
six of which lie on an ``island chain'' and are alternately elliptic and hyperbolic 
$\{h_1,e_1,h_2,e_2,h_3,e_3\}$, which in some sense enclose another elliptic fixed point $e_4$.  

\begin{figure}[htbp]
\psfrag{e1}{\begin{footnotesize}$e_1$\end{footnotesize}}
\psfrag{e2}{\begin{footnotesize}$e_2$\end{footnotesize}}
\psfrag{e3}{\begin{footnotesize}$e_3$\end{footnotesize}}
\psfrag{e4}{\begin{footnotesize}$e_4$\end{footnotesize}}
\psfrag{h1}{\begin{footnotesize}$h_1$\end{footnotesize}}
\psfrag{h2}{\begin{footnotesize}$h_2$\end{footnotesize}}
\psfrag{h3}{\begin{footnotesize}$h_3$\end{footnotesize}}
\begin{center}
\includegraphics[scale=.33]{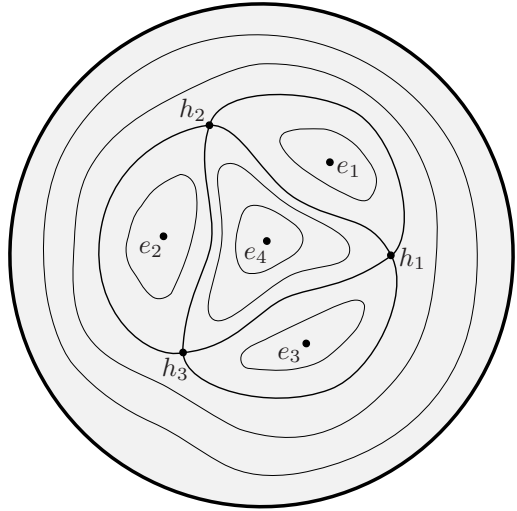}
\caption{{A configuration containing  six fixed points which encloses an elliptic fixed point.}}
\label{F:integrable_diskmap}
\end{center}
\end{figure}

No matter what the flow on $\T_1$ is, provided it generates $\psi_H$ as its first return map, the 
six fixed points on the island chain are ``locked together'' in the sense that any pair of them must 
have the same linking number.  Let us assume that the flow is arranged so that this linking number is 
zero.  In other words, for any distinct pair $x,y\in\{h_1,e_1,h_2,e_2,h_3,e_3\}$ we have $\lk(x,y)=0$, 
if we abbreviate $\lk(x,y)$ to mean the linking number of the corresponding periodic orbits in $(Z_1,\lambda_1^H)$.  Since the island chain lies ``outside'' of $e_4$, it follows that for any $x\in\{h_1,e_1,h_2,e_2,h_3,e_3\}$ 
we have $\lk(x,e_4)=0$ also.  So, since the linking number is a symmetric function of its entries, 
all seven isolated fixed points have a common linking number which is zero.  

Fix any one of the isolated elliptic fixed points, e.g. $e_1$.  Let $\gamma_{e_1}:S^1\rightarrow Z_1$ denote the 
periodic orbit in the mapping torus corresponding to the fixed point $e_1$.  Let $k$ be any integer. 
Potentially $\psi_H$ has other non-isolated fixed points if one of the 
invariant circles has rotation number an integer, in which case Theorem \ref{T:existence_of_foliations:_strongest_statement} only applies to a small perturbation.  
But to see what is going on imagine that we can nevertheless find  a transversal foliation 
$\Ftilde^k(\gamma_{e_1})$ of $\T_1$, which contains $\gamma_{e_1}$ as a spanning orbit, and for which 
the boundary condition is the integer $k$.  Then, by Lemma \ref{L:linking_number_equals_boundary_condition},  
every spanning orbit $\gamma\in\P(\Ftilde^k(\gamma_{e_1}))$ 
distinct from $\gamma_{e_1}$, (if such exists) must satisfy 
\begin{equation}\label{E:linking_numbers_for_integrable_example}
 			\lk(\gamma,\gamma_{e_1})=k.  
\end{equation}
Let us denote by 
\[
 			\P(\T_1,\lambda^H)
\]
the collection of periodic Reeb orbits in $Z_1$ that correspond to fixed points of $\psi_H$, in other 
words those which are homologous to the longitude $L_1$.  It turns out that, provided $k\neq 0$, we can 
characterize the subset of spanning orbits $\P(\Ftilde^k(\gamma_{e_1}))\subset\P(\T_1,\lambda^H)$ 
as precisely those orbits $\gamma\in\P(\T_1,\lambda^H)$ which are either $\gamma_{e_1}$ or which satisfy 
equation (\ref{E:linking_numbers_for_integrable_example})!  This is not true in general, and uses 
extremely weak properties of the integrability of $\psi_H$, and is not true for $k=0$ without more 
information about $\psi_H$.  (A more general statement along these lines, without any integrability 
assumptions, can be made, and may be the subject of a future paper.)  

We may carry out this procedure for all iterates.  The asymptotic behavior is interesting:  First, fix an 
irrational number $\omega\in\R$.  Now, pick any sequence of integers $k_n\in\Z$ such that 
\begin{equation}\label{E:omega_as_limit_of_ave_boundary_conditions}
 		\lim_{n\rightarrow\infty}\frac{k_n}{n}=\omega.
\end{equation}
Let $(Z_1,\lambda^H_1),(Z_2,\lambda^H_2,),....$ be the sequence of Reeb-like mapping tori 
obtained by lifting $(Z_1,\lambda^H_1)$, where $(Z_n,\lambda^H_n)$ generates the $n$-th iterate $\psi_H^n$.  
For each $n\in\N$, let 
\[
		 \gamma_{e_1}^n:S^1\rightarrow Z_n
\]
denote the lift of $\gamma_{e_1}$.  
By Theorem \ref{T:existence_of_foliations:_strongest_statement} we find an almost complex 
structure $J_1$ compatible with $(Z_1,\lambda^H_1)$, such that for each $n\in\N$ there exists 
a finite energy foliation $\Ftilde_n$ associated to $(\T_n,\lambda^H_n,J_n)$, where $J_n$ denotes 
the lift of $J_1$, with the following properties.  
\begin{itemize}
 \item $\gamma_{e_1}^n$ is a spanning orbit for $\Ftilde_n$.  
 \item The boundary condition for $\Ftilde_n$ is the integer $k_n$.  
\end{itemize}

Then the sequence of finite energy foliations $\Ftilde_n$ 
converges in a $C^\infty_{\textup{loc}}$-sense to a locally-finite energy foliation, we will denote by $\Ftilde_\infty$, of the symplectization of the universal covering $( Z_{\infty},\lambda^H_{\infty},J_\infty)$, 
as described at the beginning of section \ref{S:asymptotic_foliations}.    

What do the leaves of $\Ftilde_\infty$ look like?  

\begin{figure}[htbp]
\psfrag{e1}{\begin{footnotesize}$e_1$\end{footnotesize}}
\psfrag{C1}{$C_1$}
\psfrag{C2}{$C_2$}
\psfrag{C3}{$C_3$}
\begin{center}
\includegraphics[scale=.33]{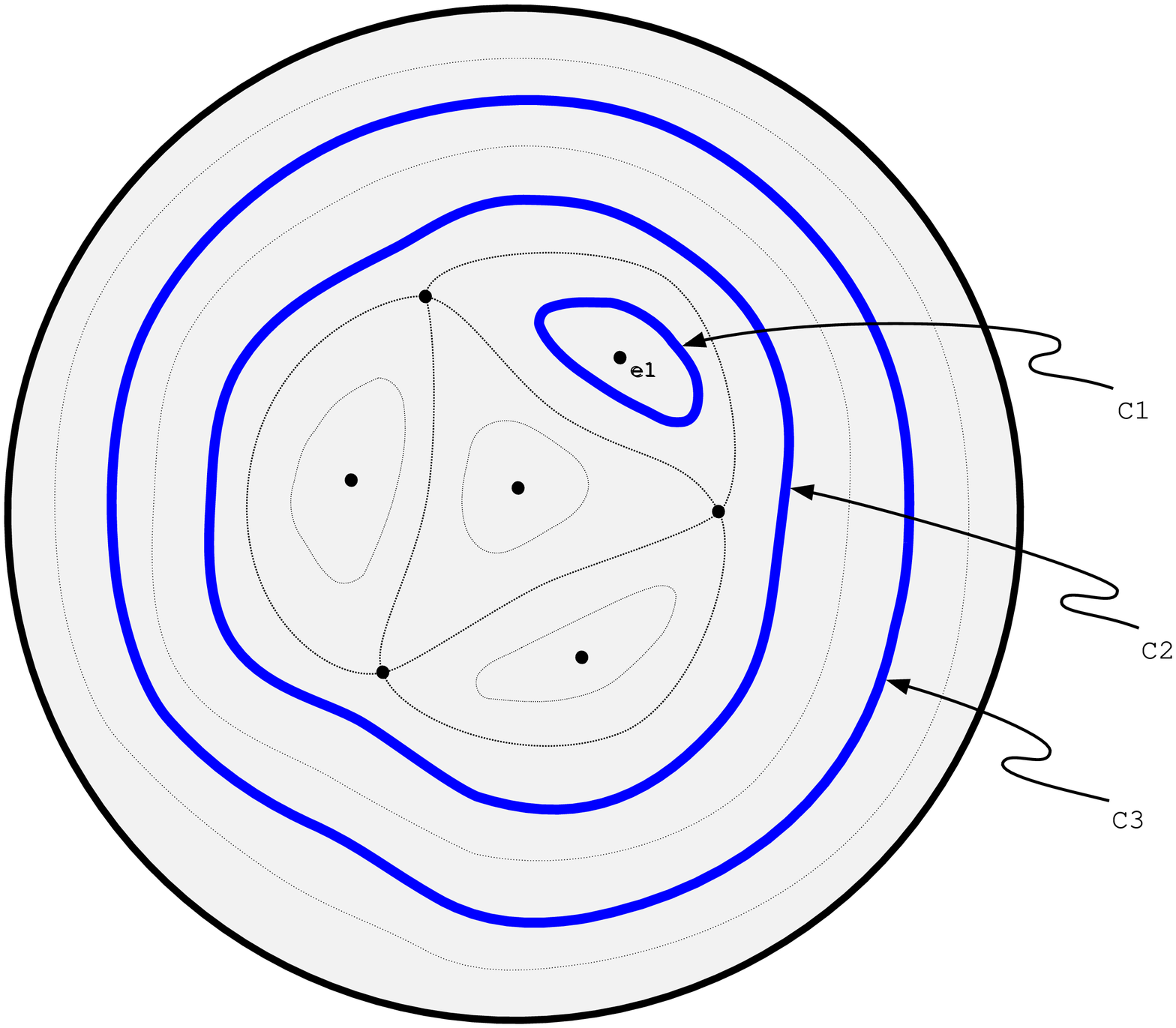}
\caption{}
\label{F:integrable_diskmap_with_some_circles_highlighted}
\end{center}
\end{figure}

Let $C_1,C_2,\ldots,C_m\subset D$ be those invariant circles $C$ of the disk map $\psi_H$ which are 
characterized by the following properties:  
\begin{itemize}
 \item When restricted to $C$, $\psi_H$ has rotation number $\omega$.  
 \item $C$ separates $e_1$ from the boundary of the disk.  That is, $e_1$ and $\partial D$ lie in 
 different components of $D\backslash C$.  
\end{itemize}
We can order these circles so that for each $j\in\{1,\ldots,m-1\}$ the circle $C_{j+1}$ is closer to the 
boundary than $C_{j}$.  Figure \ref{F:integrable_diskmap_with_some_circles_highlighted} 
illustrates a possible scenario.  Then it turns out that $\Ftilde_\infty$ looks something like 
in Figure \ref{F:limiting_foliation_for_integrable_diskmap}.  

\begin{figure}[htb]
\psfrag{e1}{\begin{footnotesize}$e_1$\end{footnotesize}}
\psfrag{R1}{$R_1$}
\psfrag{R2}{$R_2$}
\psfrag{R3}{$R_3$}
\psfrag{R4}{$R_4$}
\begin{center}
\includegraphics[scale=.33]{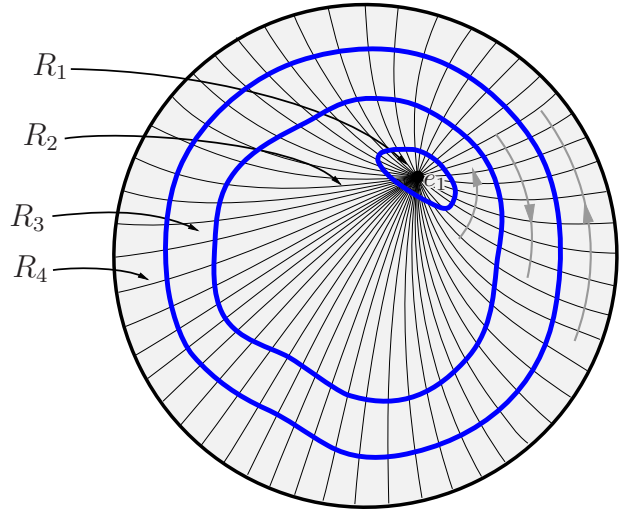}
\caption{A limiting foliation for the integrable disk map $\psi_H$ from 
figure \ref{F:integrable_diskmap}.  
The leaves with vanishing $d\lambda$-energy correspond to the quasi-periodic orbits in the 
invariant circles which have the chosen rotation number $\omega$ and which surround the chosen 
elliptic fixed point $e_1$.  Each leaf with non-zero $d\lambda$-energy is a strip connecting two such 
orbits on neighboring circles.  The flow is transverse to these strips in alternating directions 
as one passes through a circle, as indicated by the grey arrows.}  
\label{F:limiting_foliation_for_integrable_diskmap}
\end{center}
\end{figure}

More precisely, the leaves in $\Ftilde_\infty$ which have vanishing $d\lambda_\infty$-energy correspond 
precisely to these circles and the fixed point $e_1$ itself.  In other words, if $F\in\Ftilde_\infty$, 
has 
\[
 					\int_F d\lambda_\infty=0, 
\]
then $F=\R\times\phi_z(\R)$ is the plane over an orbit $\phi_z:\R\rightarrow Z_\infty$ 
which corresponds to the $\psi$-orbit of a point $z$ in $\{e_1\}\cup C_1\cup\ldots\cup C_m$.  Conversely, 
for every point 
\begin{equation*}
 		z\in\{e_1\}\cup C_1\cup\ldots\cup C_m	
\end{equation*}
the plane with image $\R\times\phi_z(\R)$, where $\phi_z$ is the Reeb orbit passing through the point 
$(0,z)\in Z_\infty$, is a leaf in $\Ftilde_\infty$.  

The remaining leaves can be described schematically as follows.   
Let $R_1,R_2,\ldots,R_m$ denote the open annular regions in the disk for which 
$R_j$ has boundary circles $C_{j-1},C_j$, for $j=1,\ldots m$, and $R_{m+1}$ is the remaining region 
between $C_m$ and $\partial D$.  (So $R_{m+1}$ includes the points in the boundary of $D$ unless $\partial D=C_m$.)  For example see Figure \ref{F:limiting_foliation_for_integrable_diskmap}.  Then each leaf in $\Ftilde_\infty$ having non-zero $d\lambda_\infty$-energy is a plane which projects 
down to an infinitely long open strip in the three-manifold $Z_\infty$.  With respect to the 
coordinates $(\z,\x,\y)$ on $Z_\infty=\R\times D$, the strip extends to plus and minus 
infinity in the $\z$-direction.   The two boundary components of the closure 
of the strip are embedded copies of $\R$.  Indeed, each is the image of a Reeb orbit; one corresponding 
to a point in some $C_j$ or $e_1$, and the other corresponding to a point in $C_{j+1}$ or $C_1$.  
Hence, each leaf with non-zero $d\lambda_\infty$-energy is in some sense asymptotic to a quasi-periodic 
orbit of $\psi_H$, with rotation number $\omega$ about the chosen fixed point $e_1$.

To summarize.  A number of interesting issues arise when attempting to generalize this approach to finding 
invariant circles or quasi-periodic behavior.  There are compactness issues when taking a limit 
of a sequence of finite energy foliations because the $E$-energy must blow up.  Before one even 
takes a limit, there are questions regarding the symmetry of finite energy foliations that also seem 
important.  For example, on the one hand it is much easier to understand symmetry of a foliation 
with respect to the $\Z$-action described in equation (\ref{D:deck_transformation}) when the map is 
integrable.  On the another hand, the integrability is largely irrelevant.  

With the success of KAM results in mind, one knows that Diophantine properties must play a role.    
It is conceivable that local KAM techniques enter the analysis and become combined 
with the pseudoholomorphic curves in this framework.  In the next section 
we describe a more concrete situation where this may be a possibility.  

\subsection{Speculations on Pseudo-Rotations}\label{S:pseudorotations}
Pseudo-rotations (also called irrational pseudo-rotations) can be defined as smooth, orientation 
preserving, area preserving diffeomorphisms of 
the unit disk which have a single fixed point and no other periodic points.  Obvious 
examples are rigid rotations with irrational rotation number, and smooth conjugacies of these.  
The question arises whether these are the only examples.  

As early as 1970 Anosov and Katok constructed ``exotic'' examples of pseudo-rotations which are 
ergodic \cite{AnosovKatok} and which therefore cannot be conjugated to a rotation.  In all their  
examples the rotation number on the boundary is a Liouville number - an irrational number well 
approximated by rationals.  Later work of Fayad and Saprykina \cite{FayadSaprykina} established 
examples for any Liouville rotation number on the boundary.  Their examples are not only ergodic 
but weak mixing.  Meanwhile, unpublished work of Herman 
precluded such ergodic pseudo-rotations when the rotation number on the boundary is Diophantine - those 
irrationals which are not Liouvillean.  

This seems to have led Herman to raise the following question at his ICM address in '98: is every 
pseudo-rotation with Diophantine rotation number on the boundary $C^\infty$-smoothly conjugate to a 
rigid rotation?  (That this is even true for circle diffeomorphisms is a deep result of his from 
'79 \cite{Herman}).  

An important step was recently taken in this direction.     
Fayad and Krikorian \cite{FayadKrikorian} proved a beautiful result, that in particular 
answered affirmatively a local version of Herman's question.  Their approach, using KAM methods, 
was apparently to some degree based on ideas of Herman, and they referred 
to the result as ``Herman's Last Geometric Theorem''.  They showed, that any pseudo-rotation 
that has Diophantine rotation number $\alpha$ on the boundary is $C^\infty$-smoothly conjugate to the rigid 
rotation $R_\alpha$, provided the disk map is already globally sufficiently close to $R_\alpha$ 
in some $C^k$-topology, where $k$ is finite and depends on $\alpha$.  

The global question is apparently still open.  Holomorphic curves, when one has them, have proven 
successful at handling global problems in symplectic geometry. Placing 
pseudo-rotations in the framework of this paper, there are plenty of holomorphic curves available, and 
we speculate that there might be something new to be gained from this angle.  We briefly 
describe where first observations lead.  

Let $\psi:D\rightarrow D$ be a pseudo-rotation, which on the boundary is a genuine rotation with 
rotation number $\alpha$.  Fix a Reeb-like mapping torus $(Z_1,\lambda_1)$ 
generating $\psi$ as its first return map.  Let $\gamma:S^1\rightarrow Z_1$ be the unique, simply 
covered, periodic orbit of the Reeb flow.  For each $n\in\N$ we have the ``longer'' mapping torus 
$(Z_n,\lambda_n)$ generating $\psi^n$, and the unique simply covered periodic orbit we denote by 
$\gamma^n$.  For each such $n$ Theorem \ref{T:existence_foliations_thesis} provides finite 
energy foliations $\Ftilde_n^-$ and $\Ftilde_n^+$ associated to $(Z_n,\lambda_n)$ which have 
boundary condition $\lfloor n\alpha\rfloor$ and $\lceil n\alpha \rceil$ respectively.  Each has 
the single spanning orbit, $\gamma^n$.  A schematic picture of $\Ftilde_n^-$ and $\Ftilde_n^+$, 
is shown in Figure \ref{F:foliations_for_pseudorotations}.  In general the arcs connecting the fixed point 
to the boundary are not perfect radial lines of course, although depicted as such in the figure.  

\begin{figure}[thb]
\psfrag{Fminus}{$\Ftilde_n^-$}
\psfrag{Fplus}{$\Ftilde_n^+$}
\begin{center}
\includegraphics[scale=.4]{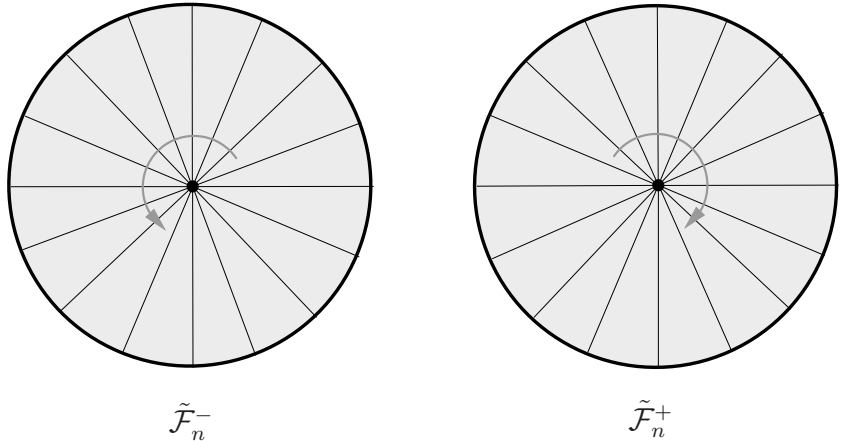}
\caption{Two finite energy foliations associated to an iterate of a pseudo-rotation.  On the 
left the trajectories pass through the leaves positively, on the right they pass through 
negatively.}
\label{F:foliations_for_pseudorotations}
\end{center}
\end{figure}

The $\R$-action on the foliations can be translated into an $\R$-action on the projected leaves 
in the three-manifold $Z_n$, which on the disk slice $D_0$ induces something akin to a ``radial'' 
coordinate on the disk.  

It turns out that each finite energy foliation $\Ftilde_n^-$ and $\Ftilde_n^+$ gives rise to a 
smooth diffeomorphism of the disk, not necessarily area preserving, 
\[
 			\varphi_n^-,\varphi_n^+:D\rightarrow D, 
\]
respectively, which fix the fixed point of $\psi^n$, are each $n$-th roots of the identity map on $D$, 
and the rotation number of $\varphi_n^-$ on the boundary is $\lfloor n\alpha\rfloor/n$, while the 
rotation number of $\varphi_n^+$ on the boundary is $\lceil n\alpha\rceil/n$.  Moreover, one has 
the following.  A similar statement for the $\varphi_n^+$ maps holds also.  
  
\begin{theorem}[Bramham]\label{T:approximating_pseudorotations_by_periodic_maps}
Let $\psi:D\rightarrow D$ be a pseudo-rotation with fixed point $p$, and coinciding with a rigid 
rotation on the boundary circle.  Let $\alpha\in\R/Z$ be the rotation number on the boundary.  
Then for each $n\in\N$ there exists a $C^\infty$-diffeomorphism $\varphi_n^-:D\rightarrow D$, 
of finite order, which fixes the point $p\in D$, for which there exists a subsequence 
$\varphi_{n_j}^-$ which converges in the $C^\infty$-topology to $\psi$.  
More precisely, $(\varphi_n^-)^n=\id_D$, and the rotation number on the boundary is the projection 
of $\lfloor n\alpha\rfloor/n$ down to $\R/\Z$.  
\end{theorem}
Thus each pseudo-rotation is the $C^{\infty}$-limit of a sequence of maps which fill the whole 
disk, minus the fixed point, with invariant circles.  
Note that no Diophantine conditions are placed on the rotation number $\alpha$ in this statement.  
Under the assumption that $\alpha$ is Diophantine one would like 
to control these circles and show that the pseudo-rotation also fills the disk with 
invariant circles.  

If Herman's question is answered affirmatively, one could also ask the following question about 
Liouvillean pseudo-rotations.  Does every pseudo-rotation with Liouvillean rotation number at least 
lie in the closure, in the $C^\infty$-topology, of the set 
\[
 			\chi=\{h\circ R_t\circ h^{-1}\,|\,t\in\R/\Z,\  h\in\Diff_{\omega_0}(D)\}.
\]
Here $R_t$ denotes the rotation through angle $2\pi t$ and $\Diff_{\omega_0}(D)$ refers to the smooth, 
area preserving, orientation preserving, diffeomorphisms.  This is apparently unknown.  If this were the case, it would mean that all 
pseudo-rotations which are conjugate to rotations lie in $\chi$, while all those which are 
not lie on the boundary of $\chi$.  Theorem \ref{T:approximating_pseudorotations_by_periodic_maps} 
suggests that the finite energy foliations might provide a way to approach this question.  
Both authors hope to explore these ideas in the future.

\subsection[Questions]{Questions about the asymptotic behavior of finite energy foliations and disk maps}
\begin{question}
Suppose that the disk map $\psi$ has a unique invariant circle, or quasi-periodic orbit, $C$ surrounding 
an elliptic fixed point $e\in D$ that has irrational rotation number $\omega$.  
If $\psi$ is the time-1 map of an autonomous Hamiltonian then we described above a method by 
which a suitably chosen sequence of finite energy foliations picks out $C$.  Does the same approach 
pick out $C$ without the global integrability assumption on $\psi$?  
\end{question}

\begin{question}
Can one, and if so how, recover results of Aubry-Mather theory using the foliations?  That is, assume 
that the disk map $\psi$ 
is a monotone twist map on the complement of the fixed point $0\in D$.  How do we find quasi-periodic orbits 
surrounding $0$ for all rotation numbers in the twist interval of $0$?  Can one do this using assumptions 
of monotone twist maps which are symplectic, that is, which remain true under any symplectic change 
of coordinates?  For example, using only say some inequalities on the Conley-Zehnder indices of period 
orbits?  
\end{question}

Suppose that $\Ftilde_j$, over $j\in\N$, is a sequence of finite energy foliations associated to $(Z_{n_j},\lambda_{n_j})$, 
where $(Z_1,\lambda_1)$ is a Reeb-like mapping torus generating $\psi:D\rightarrow D$, and $n_j$ is a sequence 
of integers tending to $+\infty$.  Let $k_j\in\Z$ be the boundary condition for $\Ftilde_j$.  

\begin{question}
Suppose that $e\in\Fix(\psi)$ is an interior fixed point, let us say elliptic.  Let 
$\gamma_e:S^1\rightarrow Z_1$ be the corresponding periodic Reeb orbit.  Suppose that for each $j$ 
there exists a closed cycle $C_j$ of rigid leaves in $\Ftilde_j$, whose projection down to $Z_{n_j}$ 
separates $\gamma_e^{n_j}$ from the boundary of $Z_{n_j}$, and satisfies: 
\begin{enumerate}
 \item The sequence $E_j:=\max\{ E_{d\lambda_{n_j}}(F)\,|\,F\in C_j\}$ converges to zero. 
 \item The ratio $k_j/n_j$ converges to some irrational number $\omega\in\Q^c$.  
 \item Each $C_j$ is invariant under the automorphism $\tau:\R\times Z_{n_j}\rightarrow\R\times Z_{n_j}$ 
where $\tau(a,z,x,y)=(a,z+1,x,y)$, meaning that for each leaf $F\in C_j$, we have $\tau(F)\in C_j$.  
\end{enumerate}
Let $S_j\subset D\backslash\{e\}$ be the unparameterized circle given by: 
\[
 	S_j:=\{ z\in D\,|\, \mbox{there exists } F\in C_j\mbox{ such that }(0,z)\in\pr(F)\}
\]
where $\pr:\R\times Z_{n_j}\rightarrow Z_{n_j}$ is the projection.  

The question then is: are there further conditions under which the sequence $S_j$ of compact connected 
subsets of $D\backslash\{e\}$ converges in some sense to a closed $\psi$-invariant subset of 
$D\backslash\{e\}$ on which all points have rotation number $\omega$ about $e$?  
\end{question}

Recall the deck transformation $\tau: Z_{\infty}\rightarrow Z_{\infty}$ given by $\tau(\z,\x,\y)=(\z+1,\x,\y)$.  
\begin{definition}
For an integer $q\in\Z$, let us say that an $\R$-invariant foliation $\Ftilde_\infty$ by 
pseudoholomorphic curves associated to $( Z_{\infty},\lambda_{\infty},J_\infty)$ is \emph{$q$-shift invariant} 
if for every leaf $F\in\Ftilde_\infty$ we have $\id\times\tau^q(F)$ is also a leaf.  
\end{definition} 
An easy observation is that if $\Ftilde_\infty$ is $q$-shift invariant, then the 
subcollection of leaves having vanishing $d\lambda_\infty$-energy picks out a $\psi^q$-invariant 
subset of the disk.  Indeed, let $\Omega\subset D$ be the set of $z\in D$ such that the plane 
\[
 		F=\R\times\phi_z(\R),  
\]
over the Reeb orbit $\phi_z:\R\rightarrow Z_{\infty}$ characterized by $\phi_z(0)=(0,z)$, is a leaf in $\Ftilde_\infty$.  Then, since $\tau^q(\R\times\phi_z(\R))=\R\times\phi_{\psi^{-q}(z)}(\R)$ is always 
true, then if $z\in\Omega$, then $\tau^q(F)\in\Ftilde_\infty$ implies that $\R\times\phi_{\psi^{-q}(z)}(\R)$ 
is a leaf in $\Ftilde_\infty$, and so $\psi^{-q}(z)\in\Omega$.  Thus $\psi^{-q}(\Omega)\subset\Omega$.  Similarly as $q$-shift invariance implies $(-q)$-shift invariance, 
$\psi^{q}(\Omega)\subset\Omega$, and so $\psi^{q}(\Omega)=\Omega$.   
\begin{question}
 Is a statement along the following lines true?  Consider a sequence $\Ftilde_n$ of 
finite energy foliations associated to Reeb-like mapping tori $(Z_1,\lambda_1),(Z_2,\lambda_2),\ldots$ 
where $(Z_1,\lambda_1)$ generates the area preserving disk map $\psi$.  Suppose that the sequence of 
foliations $\Ftilde_n$ converges in a $C_{\textup{loc}}^\infty$-sense to a foliation 
$\Ftilde_\infty$ associated to the universal covering $( Z_{\infty},\lambda_{\infty})$.  Suppose 
that $\Ftilde_\infty$ is $q$-shift invariant, for some $q\in\Z$.  Let $\Omega\subset D$ 
be the set of points $z\in D$ such that the plane
\[
 		F=\R\times\phi_z(\R),  
\]
over the Reeb orbit $\phi_z:\R\rightarrow Z_{\infty}$ characterized by $\phi_z(0)=(0,z)$, is a leaf in $\Ftilde_\infty$.   We just saw that $\psi^q$ restricts to a map on $\Omega$.  Then the question is, 
are there interesting general assumptions under which $\psi^q$ must have zero topological entropy 
on $\Omega$?  
\end{question}

Note that in the last question one needs to rule out the vertical foliation  
$\Ftilde^\nu( Z_{\infty},\lambda_{\infty})$, as $\Omega$ for this is the whole 
disk regardless of the dynamics.  Of course the 
vertical foliation is $q$-shift invariant for all $q\in\Z$, so this condition does not rule it out.   
This is another reason to phrase it so that $\Ftilde_\infty$ is achieved as a limit of finite energy 
foliations; then one could perhaps place assumptions on the sequence such as uniform bounds on some 
average notion of energy, or average of the boundary conditions, to rule out obtaining this trivial 
foliation.

\vspace{0.5cm}
\noindent {\bf Acknowledgment:} This material is based upon work supported by the National Science 
Foundation under agreement No. DMS-0635607 (BB) and by NSF DMS-1047602 (HH).  Any opinions, findings and conclusions or recommendations 
expressed in this material are those of the authors and do not necessarily reflect the views of the 
National Science Foundation.

\end{document}